\documentclass[a4paper,11pt,twoside]{article}
\usepackage{latexsym,amsfonts,index,amssymb}
\usepackage[centertags]{amsmath}
\usepackage{amstext}
\usepackage{xspace}
\usepackage{cases}
\usepackage{enumitem}
\usepackage{bbold}

\def\titlerunning#1{\gdef\titrun{#1}}
\makeatletter
\def\author#1{\gdef\autrun{\def\and{\unskip, }#1}\gdef\@author{#1}}
\def\address#1{{\def\and{\\\hspace*{18pt}}\renewcommand{\thefootnote}{}%
\footnote {#1}}%
\markboth{\autrun}{\titrun}} \makeatother
\def\email#1{e-mail: #1}
\def\subjclass#1{{\renewcommand{\thefootnote}{}%
\footnote{\emph{Mathematics Subject Classification (2010):} #1}}}
\def\keywords#1{\par\medskip
\noindent\textbf{Keywords.} #1}

\setenumerate[1]{label=$(\alph*)$,leftmargin=*}
\setenumerate[2]{label=$(\roman*)$,leftmargin=*}

\DeclareMathAlphabet{\mathpzc}{OT1}{pzc}{m}{it}

\numberwithin{equation}{section}

\def\N{{\mathbb N}}
\def\Z{{\mathbb Z}}
\newfont{\sss}{cmssi10 at 11pt}
\newfont{\bss}{cmssbx10 at 11pt}
\newfont{\tit}{cmitt10 at 11pt}

\newcommand{\A}{{\bf A}}
\newcommand{\Se}{{\bf S}}

\newcommand{\LI}{{\bf LI}}
\newcommand{\LG}{{\bf LG}}
\newcommand{\LH}{{\bf LH}}
\newcommand{\Sl}{{\bf Sl}}
\newcommand{\LSl}{{\bf LSl}}

\newcommand{\D}{{\bf D}}
\newcommand{\G}{{\bf G}}
\newcommand{\He}{{\bf H}}

\newcommand{\V}{{\bf V}}

\newcommand{\kt}{$\kappa$-term}

\newcommand{\mfrg}[3] {{#1}:{#2}\mathop{\hbox{\kern5pt$\circ$\kern-12pt\raise0.1pt\hbox
{$\longrightarrow$}}}{#3}}

\newcommand{\infee}[2]{{#1}_{#2}^{\!{\scriptscriptstyle{-\!\infty}}}}

\newcommand{\infdd}[2]{{#1}_{#2}^{\!\scriptscriptstyle{+\!\infty}}}

\newtheorem{theorem}{Theorem}[section]
\newtheorem{proposition}[theorem]{Proposition}
\newtheorem{corollary}[theorem]{Corollary}

\newtheorem{lemma}[theorem]{Lemma}
\newtheorem{claim}{Claim}

\newtheorem{definition}[theorem]{Definition}

\newtheorem{example}[theorem]{Example}
\newenvironment{definition*}{\begin{trivlist}\item[\hskip
    \labelsep{\bf Definition\quad}]}%
  {\hfill\qed\end{trivlist}}
\newenvironment{notation*}{\begin{trivlist}\item[\hskip
    \labelsep{\bf Notation\quad}]}%
  {\end{trivlist}}

  \def\qed{{\unskip\nobreak\hfil\penalty50\hskip .001pt\hbox{}%
      \nobreak\hfil
      \vrule height 1.2ex width 1.1ex depth -.1ex
      \parfillskip=0pt\finalhyphendemerits=0\medbreak}}
\newenvironment{proof}{\begin{trivlist}\item[\hskip%
     \labelsep{\bf Proof.\quad}]}%
 {\hfill\qed\rm\end{trivlist}}

  {\hfill\qed\end{trivlist}}%

\begin{document}

\titlerunning{Semigroup presentations for test local groups}

\title{Semigroup presentations for test local groups}

\author{J. C. Costa %
  \and %
  C. Nogueira %
  \and %
  M. L. Teixeira%
}

\date{March 9, 2014}

\maketitle

\address{ %
  J. C. Costa \& M. L. Teixeira: %
  CMAT, Dep.\ Matem\'{a}tica e Aplica\c{c}\~{o}es, Universidade do Minho, Campus
  de Gualtar, 4700-320 Braga, Portugal; %
  \email{jcosta@math.uminho.pt, mlurdes@math.uminho.pt} %
  \and %
  C. Nogueira: %
   CMAT, Escola Superior de Tecnologia e Gest\~ao,
  Instituto Polit\'ecnico de Leiria,
  Campus 2, Morro do Lena, Alto Vieiro, 2411-901
   Leiria, Portugal; %
   \email{conceicao.veloso@ipleiria.pt} %
}


\begin{abstract}
  In this paper we exhibit a type of semigroup
presentations which determines a class of local groups. We show that
the finite elements of this class generate the pseudovariety ${\bf
LG}$ of all finite local groups and use them as test-semigroups to
prove that ${\bf LG}$ and $\Se$, the pseudovariety of all finite
semigroups, verify the same $\kappa$-identities involving
$\kappa$-terms of rank at most 1, where $\kappa$ denotes the
implicit signature consisting of the multiplication and the
$(\omega-1)$-power.

\end{abstract}

\section{Introduction}
Let $A$ be an alphabet. We denote by  $A^*$ the free monoid  of all finite words in $A$ (including the empty
word $1$) and by $A^+$ the free semigroup of all finite non-empty words in $A$. A \textit{semigroup
presentation} is an ordered pair $\langle A\mid R\rangle$, with $R\subseteq A^+\times A^+$. A semigroup $S$ is
said to be defined by the presentation $\langle A\mid R\rangle$ if $S$ is isomorphic to the quotient semigroup
$A^+/\rho_R$, where $\rho_R$ is the smallest congruence on $A^+$ containing $R$.  If $(u,v)\in\rho_R$ then $u$
and $v$ \textit{represent} the same element of $S$ and so it is usual to denote $(u,v)$ by $u=v$. However, when
we want to avoid ambiguity, we denote $u\equiv v$ to assure that $u$ and $v$ are precisely the same word over
$A$. See~\cite{Lallement:1979} for an introduction to this topic.

A local group is a semigroup $S$ such that $eSe$ is a group for each
idempotent $e$ of $S$. The finite local groups form a
 pseudovariety (i.e., a class of finite semigroups closed
under taking subsemigroups, homomorphic images and finite direct products) usually denoted by ${\bf LG}$. More
generally, if $\He$ is a pseudovariety of groups then ${\bf LH}$ denotes the pseudovariety of all finite
semigroups $S$ such that $eSe\in\He$ for each idempotent $e$ of $S$, and we recall that $\LH={\bf
H*D}$~\cite{Straubing:1985} where $\D$ is the pseudovariety of all finite semigroups whose idempotents are right
zeros. It is well known (see~\cite{Straubing:1985} for a proof) that a finite semigroup $S$ is a local group if
and only if all the idempotents of $S$ lie in the minimal ideal of $S$. A proof of the generalization of this
result to arbitrary semigroups can be found in Proposition~\ref{prop:caracterization_local_group} below. There,
a semigroup $S$ is characterized as being a local group if and only if $S$ has no idempotents or $S$ has a
minimal ideal $J$ which is a completely simple semigroup that contains  all the idempotents of $S$. In this
case, by the Rees-Suschkewitsch Theorem, $J$ is isomorphic to a Rees matrix semigroup over a group (the maximal
subgroup of $S$). In~\cite{Howie&Ruskuc:1994}, Howie and Ru\v{s}kuc showed how to find a semigroup presentation
for a Rees matrix semigroup $\mathcal{M}[G;I,\Lambda;P]$ given a semigroup presentation for the group $G$.

 In~\cite{Almeida&Azevedo:1993} (see
also~\cite[Section 10.6]{Almeida:1992}), Almeida and Azevedo showed
that a semidirect product of the form $\V*\D$, with the
pseudovariety $\V$ not locally trivial, is generated by a class
formed by certain semigroups $M_k(S,{\hbar})$ with $k\geq 1$,
$S\in\V$, $A$ an alphabet and ${\hbar}:A^+\rightarrow S$ a
$k$-superposition homomorphism. Therefore, possible properties of
$\V*\D$ may be tested on the semigroups $M_k(S,{\hbar})$ and
Almeida and Azevedo applied those test-semigroups (an expression
used in~\cite{Almeida:1992}) to obtain a representation of the free
pro-$(\V*\D)$ semigroup over $A$.

In this paper, we introduce a class of local groups
$\mathcal{S}(G,L,{\mathsf f})$ with $G$ a group, $L\subseteq A^+$ a
factorial language and ${\mathsf f}:L\cup \ddot{L}\rightarrow G$ a
function where, informally speaking, $\ddot{L}$ is formed by the
words over $A$ with minimal length that do not belong to $L$. The
group $G$ is the maximal subgroup of $\mathcal{S}(G,L,{\mathsf f})$,
 $L$ is the set of non-regular elements of $\mathcal{S}(G,L,{\mathsf f})$ and ${\mathsf f}$
 serves to define the operation of the semigroup. Given a semigroup presentation for $G$, we
 describe a semigroup presentation for $\mathcal{S}(G,L,{\mathsf f})$. In particular, when $G$ is finitely
presented, the semigroup presentation for $\mathcal{S}(G,L,{\mathsf f})$ is finite if and only if $L$ is finite.
For  a finite group $G$, the semigroup $M_k(G,{\hbar})$ is a local group. We show that, when $L_k$ is the
language of all words over $A$ of length at most $k$ and ${\mathsf f}$  is  the restriction of ${\hbar}$ to
$L\cup \ddot{L}$, $M_k(G,{\hbar})$ is a subsemigroup of a homomorphic image of $S_k(G,{\mathsf
f})=\mathcal{S}(G,L_k,{\mathsf f})$. As a consequence of the above mentioned results of Almeida and Azevedo, we
deduce that the semigroups $S_k(G,{\mathsf f})$, with $k\geq 1$ and $G\in\He$, form a generating set of the
pseudovariety $\LH$. Obviously, the more general class of semigroups $\mathcal{S}(G,L,{\mathsf f})$, with
$G\in\He$ and $L$ finite, generates $\LH$. While the semigroups $S_k(G,{\mathsf f})$ and $M_k(G,{\hbar})$ have essentially the same
capabilities as test-semigroups for ${\bf LH}$, the semigroups $\mathcal{S}(G,L,{\mathsf f})$
have the advantage that one may explore the possibility of choosing appropriate languages and functions
${\mathsf f}$ to test specific properties of ${\bf
 LH}$. This makes the semigroups
$\mathcal{S}(G,L,{\mathsf f})$ interesting
 alternatives to the semigroups $M_k(G,{\hbar})$.

 We will use the semigroups $\mathcal{S}(G,L,{\mathsf f})$ to show
that ${\bf LG}$ and $\Se$, the pseudovariety of all finite
semigroups, satisfy the same identities involving $\kappa$-terms of
rank at most 1 (i.e., terms obtained from ones of the form $u$ and
$v^{\omega-1}$, with $u\in A^*$ and $v\in A^+$, by finite
concatenation) and that these identities are decidable over $\LG$
(and $\Se$). The semigroups $\mathcal{S}(G,L,{\mathsf f})$  will be
employed in a more general context
in~\cite{Costa&Nogueira&Teixeira:2013} to solve the word problem for
identities involving arbitrary $\kappa$-terms over $\LG$. Recall
that this type of word problem is already solved, for instance, for
the pseudovarieties $\LI$ of locally trivial
semigroups~\cite{Almeida:1992} and $\LSl$ of local
semilattices~\cite{Costa:2001}. Althought the pseudovariety $\Sl$ of
finite semilattices is not a pseudovariety of groups, the equality
$\LSl=\Sl*\D$ holds and the representations of free pro-$(\Sl*\D)$
semigroups, obtained by Almeida and Azevedo, were used to solve the
above mentioned word problem as well as to prove other properties of
$\LSl$~\cite{Costa&Teixeira:2004,Costa&Nogueira:2009}.

\section{Semigroup presentations for a class of local groups}\label{Presentations-local-groups}
 In this section we introduce a certain type of semigroup presentation
 and prove that the semigroups they define are local
 groups. For an introduction to combinatorics on words, the reader is
referred to~\cite{Lothaire:2002}.

We begin by giving a characterization of local groups. Recall that a semigroup is \emph{simple} if it has a
unique ideal. The set of idempotents of a semigroup $S$, denoted by $E(S)$, is endowed with a natural (partial)
order relation defined by the rule that $e\leq f$ if and only if $ef=fe=e$. A simple semigroup $S$ is said to be
\emph{completely simple} if $E(S)$ has minimal elements for the relation $\leq$.
\begin{proposition}\label{prop:caracterization_local_group}
A semigroup $S$ is a local group if and only if $E(S)=\emptyset$ or $S$ has a completely simple minimal ideal
containing $E(S)$.
\end{proposition}
\begin{proof} Let $S$ be a local group and suppose that $E(S)$ is non-empty. Denote by $J_a$
the principal ideal $S^1aS^1$ generated by an element $a$ of $S$. Let $e\in E(S)$. Then, by definition of local
group, the local submonoid $eSe$ is a group with identity $e$. Hence, for each $a\in S$ there exists $a'\in S$
such that $(eae)(ea'e)=e$. This means that, for all $a\in S$, $e\in J_a$ and so $J_e\subseteq J_a$. Since $e$ is
an arbitrary idempotent, it follows that $J=J_e$ is a minimal ideal of $S$ containing $E(S)$. We now claim that
every element of $E(S)$ is minimal for the relation $\leq$. Indeed, let $e,f\in E(S)$ and suppose that $e\leq
f$, that is, suppose that $ef=fe=e$. Then $fef=e$, whence $e$ belongs to the group $fSf$ with identity $f$. The
fact that $e$ is an idempotent shows that $e=f$ and proves the claim. We conclude that $J$ is a completely
simple semigroup.

Reciprocally, if $E(S)=\emptyset$ then $S$ is (trivially) a local group. Suppose now that $E(S)$ is non-empty,
that $S$ has a completely simple minimal ideal $J$ and that $E(S)$ is a subset of $J$. In particular, by the
Rees-Suschkewitsch Theorem, $J$ is isomorphic to a Rees matrix semigroup over a group $G$. Moreover, since all
the idempotents of $S$ are in $J$, one deduces that each local submonoid $eSe$ is $H_e$, the $\mathcal{H}$-class
of $e$, a group isomorphic to $G$, thus proving that $S$ is a local group.
\end{proof}

Consider a word $w=a_1a_2\ldots a_n\in A^+$ ($a_i\in A$) of length
$n\geq 1$. For $1\leq p\leq q\leq n$, we denote
$w[p,q]=a_pa_{p+1}\cdots a_q$. A word $u\in A^+$ is a  (non-empty)
\emph{factor} of  $w$
 if $u=w[p,q]$ for some $p$ and $q$. In this case $w[p,q]$ is said to be an
 {\em occurrence} of the factor $u$ in $w$. We will say also  ``an occurrence $u=w[p,q]$ in
 $w$'' instead of ``an occurrence $w[p,q]$ of $u$ in $w$''. If there is an
occurrence $u=w[p,q]$  with $p=1$ (resp.\ $q=n$), then $u$ is called
a \emph{prefix}
 (resp.\ a \emph{suffix}) of $w$. We denote by $w_{\alpha}$ and $w_{\omega}$,
respectively, the prefix $w[1,n-1]$ and the suffix $w[2,n]$ of $w$
of length $n-1$. A language $L\subseteq A^+$ is said to be {\em factorial}
if it is closed under taking non-empty factors. Let
  $L$ be a non-empty factorial language and let
$$\ddot{L}=\{v\in A^+:\mbox{$v\not\in
L$ and $v_{\alpha},v_{\omega}\in L$}\}.$$ We assume that $L$ has content $A$ (i.e., $A\subseteq L$) and observe
that $L$ and $\ddot{L}$ are disjoint languages. Observe also the following elementary facts
\begin{eqnarray}
\forall v_1,v_2\in \ddot{L}& (\mbox{$v_1$ is a factor of
$v_2$}\;\Leftrightarrow\;v_1=v_2);\label{eq:ddotL-1}\\
\forall w\in A^+&(\ddot{L}\cap F(w)=\emptyset
\;\Leftrightarrow\;w\in L).\label{eq:ddotL-2}
\end{eqnarray}
From~\eqref{eq:ddotL-2} it follows that $\ddot{L}=\emptyset$ if and
only if $L=A^+$. We associate to any given word $w=a_1a_2\ldots
a_n\in A^+$ ($a_i\in A$) a well-determined finite sequence
$$\mbox{sc}_L[w]=(w_0,\ddot w_1,w_1,\ldots,\ddot w_m,w_m),$$ called
the {\em sequence of coordinates of $w$ determined by $L$}, as
follows:
\begin{itemize}
\item $m\geq 0$ is the number, called the
$\ddot L$-length of $w$, of occurrences of elements of $\ddot{L}$ in
$w$. Observe that $m=0$ if and only if  $w\in L$, in which case
$\mbox{sc}_L[w]=(w_0)$.

\item if $m>0$ then $\ddot w_1=w[p_1,q_1],\ldots ,\ddot w_m=w[p_m,q_m]$ are the successive occurrences of factors of $w$ that belong to
$\ddot{L}$. Notice that, for every $i$, $p_i<q_i$  and,
by~\eqref{eq:ddotL-1}, the integer interval $[p_i,\ldots,q_i]$ is contained in $[p_j,\ldots,q_j]$ for some $j$ if and only if $i=j$.

\item if $m>0$ then $w_0=w[1,q_1-1]$, $w_m=w[p_m+1, n]$ and $w_i=w[p_i+1,q_{i+1}-1]$ for
$0<i<m$. We note that $w_0a_{q_1}=w[1,q_1]$ does not belong to $L$ (since neither does $w[p_1,q_1]$ and $L$ is
factorial), whence $w_0$ is the longest prefix of $w$ in $L$. Analogously, $w_m$ is the longest suffix of $w$
that belongs to $L$. Moreover, for $0<i<m$, the factors  $w_ia_{q_{i+1}}=w[p_i+1,q_{i+1}]$ and
$a_{p_i}w_i=w[p_i,q_{i+1}-1]$ do not belong to $L$. We then say that $w_0, w_1,\ldots, w_m$ are \emph{maximal}
factors of $w$ in $L$. With this sense of maximality, one may verify that $w_0, w_1,\ldots, w_m$ is the sequence
of all maximal factors of $w$ in $L$ by the order they occur in $w$.
\end{itemize}

Alternatively, the sequence $\mbox{sc}_L[w]$ may be constructed by
the recursive application of the two following steps:
\begin{itemize}
\item If $w\in L$ then  let $\mbox{sc}_L[w]=(w)$.

\item If $w\not\in L$ then select an occurrence $\ddot v=w[p,q]$ in $w$ of a factor $\ddot v\in \ddot L$ and let
$\mbox{sc}_L[w]=\mbox{sc}_L[z_1](\ddot v)\mbox{sc}_L[z_2]$, where
$z_1=w[1,q-1]$ and $z_2=w[p+1,n]$.
\end{itemize}

\begin{example}\label{examp:langL} For $A=\{a,b\}$, consider the  factorial languages $L_1=\{a,b,a^2,ab,b^2,a^3,a^2b\}$
 and $L_2=\{a^i,a^jba^k:i>0,j\geq 0,k\geq 0\}$ over $A$. Then $\ddot{L}_1=\{ba,ab^2,b^3,a^4,a^3b\}$ and $\ddot{L}_2=\{ba^jb:j\geq 0\}$.
 For the words $u=a^5ba^2$ and $v=ab^3aba^4b$, we have
 $$\begin{array}{l}
 \mbox{sc}_{L_1}[u]=(a^3,a^4,a^3,a^4,a^3,a^3b,a^2b,ba,a^2),\\ \mbox{sc}_{L_2}[u]=(u),\\
 \mbox{sc}_{L_1}[v]=(ab,ab^2,b^2,b^3,b^2,ba,ab,ba,a^3,a^4,a^3,a^3b,a^2b),\\
\mbox{sc}_{L_2}[v]=(ab,b^2,b,b^2,ba,bab, aba^4,ba^4b,a^4b).
  \end{array}$$
\end{example}

 The above construction may be
 reverted. More precisely, given $\mbox{sc}_L[w]$ there is a deterministic procedure to calculate the word $w$.
 This shows in particular that
\begin{equation}\label{eq:injectivity_scL}
\forall w,z\in A^+\ (\mbox{sc}_L[w]=\mbox{sc}_L[z]\;\Leftrightarrow\;w=z).
\end{equation}
 The procedure is the following. If $\mbox{sc}_L[w]=(w_0)$, then
 $w=w_0$. Suppose now that $\mbox{sc}_L[w]=(w_0,\ddot w_1,w_1,\ldots,\ddot w_m,w_m)$ with $m\geq 1$ and assume that is
 possible to determine a word $z$ of $\ddot L$-length $m-1$ given its sequence of coordinates $\mbox{sc}_L[z]$. By the above construction,
 $w_0=w'(\ddot w_1)_\alpha$ for some  $w'\in A^*$ and the sequence $(w_1,\ddot w_2,\ldots,\ddot w_m,w_m)$ is precisely
 $\mbox{sc}_L[z]$ where $z=(\ddot w_1)_\omega w''$ is such that $w=w'\ddot w_1w''$. Since $w_0$ and $\ddot w_1$ are given, we may determine $w'$.
  On the other hand, by hypothesis, $z$ is calculable and, so, also is $w''$. Therefore $w$ is calculable.

Let $\langle A_G\mid R_G\rangle$ be a semigroup presentation  for a
group $G$, so that $G\cong A_G^+/\rho_{R_G}$. For simplicity of
notation, we will usually regard a given word $w\in A_G^+$ as the
element of $G$ it represents. On the other hand, by choosing for
each element $g\in G$ a word of $A_G^+$ representing $g$  we may
view $G$ as a subset of $A_G^+$. In particular, we denote by $e$ a
word of $A_G^+$ representing $1_G$, whence
\begin{equation}\label{eq:e_neutral element}
\forall w\in A_G^+,\mbox{ $w=ew=we$ in $G$.}
\end{equation}
Let $X=A\cup A_G$, let $L^1=L\cup\{1\}$ and let ${\mathsf f}:L\cup \ddot{L}\rightarrow G$ be a function. We
associate to ${\mathsf f}$ four new functions $\check{\mathsf f}:A^+\rightarrow L\cup LGL$, $\acute{\mathsf
f}:A^*\rightarrow L^1\cup GL$, $\grave{\mathsf f}:A^*\rightarrow L^1\cup LG$ and $\hat{\mathsf f}:A^*\rightarrow
G$ defined as follows: $\acute{\mathsf f}(1)=\grave{\mathsf f}(1)=1$ and $\hat{\mathsf f}(1)=1_G$; if $w\in L$
then $\check{\mathsf f}(w)=\acute{\mathsf f}(w)=\grave{\mathsf f}(w)=w$ and $\hat{\mathsf f}(w)={\mathsf f}(w)$;
if $w\in A^+\setminus L$ and $\mbox{sc}_L[w]=(w_0,\ddot w_1,w_1,\ldots,\ddot w_m,w_m)$, then
$$\begin{array}{rl}
\check{\mathsf f}(w)&\hspace*{-3mm}= w_0{\mathsf f}(\ddot
w_1){\mathsf f}(w_1){\mathsf f}(\ddot w_2)\cdots {\mathsf f}(\ddot
w_m)w_m,\\
\acute{\mathsf f}(w)&\hspace*{-3mm}= {\mathsf f}(w_0){\mathsf
f}(\ddot w_1){\mathsf
f}(w_1){\mathsf f}(\ddot w_2)\cdots  {\mathsf f}(\ddot w_m)w_m,\\
\grave{\mathsf f}(w)&\hspace*{-3mm}= w_0{\mathsf f}(\ddot
w_1){\mathsf f}(w_1){\mathsf f}(\ddot w_2)\cdots {\mathsf f}(\ddot
w_m){\mathsf f}(w_m),\\
\hat{\mathsf f}(w)& \hspace*{-3mm}= {\mathsf f}(w_0){\mathsf
f}(\ddot w_1){\mathsf f}(w_1){\mathsf f}(\ddot w_2)\cdots  {\mathsf
f}(\ddot w_m){\mathsf f}(w_m).
\end{array}$$
Now, we let $\check{\mathsf f}:X^+\rightarrow L\cup L^1GL^1$ be the
extension to $X^+$ of the above function $\check{\mathsf f}$ by
setting $\check{\mathsf f}(w)=\grave{\mathsf f}(u_0)g_1\hat{\mathsf
f}(u_1)\cdots g_{n-1}\hat{\mathsf f}(u_{n-1}) g_n\acute{\mathsf
f}(u_n)$ for each word $w=u_0g_1u_1\cdots g_nu_n\in X^+\setminus
A^+$, with $u_0,u_n\in A^*$, $u_1,\ldots,u_{n-1}\in A^+$ and
$g_1,\ldots,g_n\in A_G^+$.

It is worth observing the following properties of the functions
$\hat{\mathsf f}$ and $\check{\mathsf f}$.

\begin{lemma}\label{lemma:property_hat_function}
Let $w,z\in A^+$ and let $z_0\in L$ be the first coordinate of $z$
and $w_m\in L$ be the last coordinate of $w$ determined by $L$.
Then,
\begin{enumerate}
  \item \label{eq:property_hat_function}  $\hat{\mathsf f}(wz)=\hat{\mathsf
f}(wz_0){\mathsf f}(z_0)^{-1}\hat{\mathsf f}(z)=\hat{\mathsf
f}(w){\mathsf f}(w_m)^{-1}\hat{\mathsf f}(w_mz)$;
  \item \label{eq:property_check_function}
$ \check{\mathsf f}(wz)=\check{\mathsf f}(\check{\mathsf
f}(w)\check{\mathsf f}(z))$.
\end{enumerate}
\end{lemma}
 \begin{proof}  Let
 $$\begin{array}{rl}
 \mbox{sc}_L[w] &\hspace*{-2mm}= (w_0,\ddot w_1,w_1,\ldots,\ddot w_m,w_m), \\
 \mbox{sc}_L[z] &\hspace*{-2mm}= (z_0,\ddot z_1,z_1,\ldots,\ddot z_n,z_n),\\
  x &\hspace*{-2mm}= w_mz_0,\\
   \mbox{sc}_L[x] &\hspace*{-2mm}= (x_0,\ddot x_1,x_1,\ldots,\ddot x_p,x_p).
\end{array}$$
 Hence
  $$\begin{array}{rl}
 \mbox{sc}_L[wz_0] &\hspace*{-2mm}=(w_0,\ddot w_1,w_1,\ldots,\ddot
w_m,x_0,\ddot x_1,x_1,\ldots,\ddot x_p,x_p), \\
 \mbox{sc}_L[wz] &\hspace*{-2mm}= (w_0,\ddot w_1,w_1,\ldots,\ddot w_m,x_0,\ddot
x_1,x_1,\ldots,\ddot x_p,x_p ,\ddot z_1,z_1,\ldots,\ddot z_n,z_n).
\end{array}$$
Then,
$$\begin{array}{rl}
\hat{\mathsf f}(wz)&\hspace*{-2mm}= {\mathsf
f}(w_0){\mathsf f}(\ddot w_1){\mathsf f}(w_1)\cdots {\mathsf
f}(\ddot w_m){\mathsf f}(x_0){\mathsf f}(\ddot x_1){\mathsf
f}(x_1)\cdots {\mathsf f}(\ddot x_p){\mathsf f}(x_p) {\mathsf
f}(\ddot z_1){\mathsf f}(z_1)\cdots {\mathsf f}(\ddot
z_n){\mathsf f}(z_n)\\
&\hspace*{-2mm}=  \hat{\mathsf f}(wz_0){\mathsf
f}(z_0)^{-1} {\mathsf f}(z_0){\mathsf f}(\ddot z_1){\mathsf
f}(z_1)\cdots {\mathsf f}(\ddot
z_n){\mathsf f}(z_n)\\
&\hspace*{-2mm}= \hat{\mathsf f}(wz_0){\mathsf
f}(z_0)^{-1}\hat{\mathsf f}(z).
\end{array}$$
One can show analogously that  $\hat{\mathsf f}(wz)=\hat{\mathsf
f}(w){\mathsf f}(w_m)^{-1}\hat{\mathsf f}(w_mz)$, thus concluding
the proof of~\ref{eq:property_hat_function}.

If $m=n=0$ then  $w,z\in L$ and~\ref{eq:property_check_function} is
trivially verified. When $m,n>0$, we have
$$\begin{array}{rl}
\check{\mathsf f}(\check{\mathsf f}(w)\check{\mathsf f}(z))
&\hspace*{-2mm}=\check{\mathsf f}(w_0{\mathsf f}(\ddot w_1){\mathsf f}(w_1)\cdots
{\mathsf f}(\ddot w_m)w_mz_0{\mathsf f}(\ddot z_1){\mathsf
f}(z_1)\cdots {\mathsf f}(\ddot
z_n)z_n)\\
&\hspace*{-2mm}= w_0{\mathsf f}(\ddot w_1){\mathsf f}(w_1)\cdots {\mathsf
f}(\ddot w_m)\hat{\mathsf f}(w_mz_0){\mathsf f}(\ddot z_1){\mathsf
f}(z_1)\cdots {\mathsf f}(\ddot
z_n)z_n\\
&\hspace*{-2mm}=w_0{\mathsf f}(\ddot w_1){\mathsf f}(w_1)\cdots {\mathsf f}(\ddot
w_m){\mathsf f}(x_0){\mathsf f}(\ddot x_1){\mathsf f}(x_1)\cdots
{\mathsf f}(\ddot x_p){\mathsf f}(x_p) {\mathsf f}(\ddot
z_1){\mathsf f}(z_1)\cdots {\mathsf f}(\ddot
z_n)z_n\\
&\hspace*{-2mm}= \check{\mathsf f}(wz).
\end{array}$$
If $m=0$ and $n>0$ then $\mbox{sc}_L[wz]=(x_0,\ddot
x_1,x_1,\ldots,\ddot x_p,x_p ,\ddot z_1,z_1,\ldots,\ddot z_n,z_n)$
and the equality $\check{\mathsf f}(wz)=\check{\mathsf
f}(\check{\mathsf f}(w)\check{\mathsf f}(z))$ is checked as above.
The case  $m>0$ and $n=0$ is symmetric.
\end{proof}

 For latter reference, we state the
following extension of
Lemma~\ref{lemma:property_hat_function}~\ref{eq:property_check_function},
\begin{equation}\label{eq:property_check_function2}
\forall w,z\in X^+,\ \check{\mathsf f}(wz)=\check{\mathsf
f}(\check{\mathsf f}(w)\check{\mathsf f}(z)),
\end{equation}
whose validity may be verified by the reader.

 We finally set up the
presentations for our local groups. For each $u\in L$ and $\ddot v\in
\ddot{L}$, we define the following relations over $X$
\begin{eqnarray}
r_{u}:\hspace*{8mm}eue&\hspace*{-1.5mm}=\hspace*{-1.5mm}&{\mathsf f}(u)\nonumber\\[-1mm]
r_{\ddot v}:\hspace*{11.5mm}\ddot
v&\hspace*{-1.5mm}=\hspace*{-1.5mm}&\ddot v_{\alpha}{\mathsf
f}(\ddot v)\ddot v_{\omega}\nonumber
\end{eqnarray}
and set $R_{\mathsf f}=\{r_{u},r_{\ddot v}:u\in L,\ddot v\in \ddot{L}\}$. Denote by $\mathcal{T}[G,L,{\mathsf
f}]$ the semigroup $T$ defined by the presentation $\langle X\mid R\rangle$, where $R=R_G\cup R_{\mathsf f}$.
Taking~\eqref{eq:e_neutral element} into account, it can be shown that  the following relations hold in $T$ for any $w\in
A^+$ and $z\in A^+\setminus L$
\begin{equation}\label{extra}
e\check{\mathsf f}(w)e=  \hat{\mathsf f}(w),\quad  e \check{\mathsf f}(z) =\acute{\mathsf f}(z) , \quad
\check{\mathsf f}(z) e = \grave{\mathsf f}(z).
\end{equation}

A more important relation valid in $T$ is revealed in the next
lemma.
\begin{lemma}\label{lemma:w_evaluated_in_T}
If $w\in X^+$ is an arbitrary word, then $w=\check{\mathsf f}(w)$ in
$T$.
\end{lemma}
\begin{proof} We consider first the case where $w\in A^+$ and let $\mbox{sc}_L[w]=(w_0,\ddot
w_1,w_1,\ldots,\ddot w_m,w_m)$. We prove the result by induction on  the $\ddot L$-length $m$ of $w$. If $m=0$,
then $w=w_0\in L$ and so $\check{\mathsf f}(w)=w$ by definition of $\check{\mathsf f}$. Suppose now that $m\geq
1$ and assume, by induction hypothesis, that the result is valid for words $z\in A^+$ with $\ddot L$-length
$m-1$. Write $w=w'\ddot w_1w''$ where the first occurrence of $\ddot w_1$ is distinguished in the factorization.
Applying relation $r_{\ddot w_1}$ to that occurrence of $\ddot w_1$ one deduces that $T$ verifies $w=w'(\ddot
w_1)_\alpha {\mathsf f}(\ddot w_1)(\ddot w_1)_\omega w''$.  Moreover $w_0=w'(\ddot w_1)_\alpha$ and
$\mbox{sc}_L[z]=(w_1,\ddot w_2,w_2,\ldots,\ddot w_m,w_m)$ where $z=(\ddot w_1)_\omega w''$, whence $T$ satisfies
$w=w_0{\mathsf f}(\ddot w_1)z$.   If $m=1$, then $z=w_1$ and thus $T$ verifies $w=\check{\mathsf f}(w)$. Suppose
next that $m>1$ and notice that, by induction hypothesis, $z=\check{\mathsf f}(z)$ in $T$. So, $T$ verifies
$w=w_0{\mathsf f}(\ddot w_1)\check{\mathsf f}(z)=w_0{\mathsf f}(\ddot w_1)w_1{\mathsf f}(\ddot w_2){\mathsf
f}(w_2){\mathsf f}(\ddot w_3)\cdots  {\mathsf f}(\ddot w_m)w_m$.
 Then, by~\eqref{eq:e_neutral element}, the semigroup $T$ verifies $w=w_0{\mathsf f}(\ddot w_1)ew_1e{\mathsf
f}(\ddot w_2){\mathsf f}(w_2)\cdots  {\mathsf f}(\ddot w_m)w_m$ and also, applying relation $r_{w_1}$,
$w=w_0{\mathsf f}(\ddot w_1){\mathsf f}(w_1){\mathsf f}(\ddot w_2){\mathsf f}(w_2)\cdots {\mathsf f}(\ddot
w_m)w_m=\check{\mathsf f}(w)$, thus concluding the inductive step and the proof of the lemma for $w\in A^+$.

Suppose now that $w=u_0g_1u_1\cdots g_nu_n\in X^+\setminus A^+$,
with $u_0,u_n\in A^*$, $u_1,\ldots,u_{n-1}\in A^+$ and
$g_1,\ldots,g_n\in A_G^+$. So, in $T$,
\begin{alignat*}{4}
w & =\check{\mathsf f}(u_0)g_1\check{\mathsf f}(u_1) g_2
\check{\mathsf f}(u_2)\cdots g_n\check{\mathsf f}(u_n)& \quad &
\mbox{by the first case}\\
& = \check{\mathsf f}(u_0)eg_1e\check{\mathsf f}(u_1)e g_2
e\check{\mathsf f}(u_2)\cdots e g_n e\check{\mathsf f}(u_n)   &&
by~\eqref{eq:e_neutral element}\\
& =  \grave{\mathsf f}(u_0)g_1\hat{\mathsf f}(u_1) g_2 \hat{\mathsf
f}(u_2)\cdots  g_n \acute{\mathsf f}(u_n)  && by~\eqref{extra}\\
&=  \check{\mathsf f}(w).&&
\end{alignat*}
This concludes the proof of the lemma.
\end{proof}

Let $\psi$  be the canonical epimorphism from $X^+$ onto
$X^+/\rho_R=T$ and denote by $Z$ the subset $L\cup L^1GL^1$ of
$X^+$. By Lemma~\ref{lemma:w_evaluated_in_T}, the word
$\check{\mathsf f}(w)\in Z$ is a representative of the element
$\psi(w)\in T$. We show next that $Z$ contains exactly one
representative of each element of $T$ and  endow $Z$ with a
(natural) structure of semigroup that makes it isomorphic to $T$.
Before that we mention that, obviously, $Z= \check{\mathsf f}(X^+)$
and $\check{\mathsf f}(z)=z$ for every $z\in Z$, whence
$\check{\mathsf f}\circ \check{\mathsf f}=\check{\mathsf f}$.
\begin{proposition}\label{prop:Z_local_group}
Let $Z$ be endowed with the operation defined by $z_1\cdot
z_2=\check{\mathsf f}(z_1z_2)$ and let $J$ be the Rees matrix semigroup
$\mathcal{M}[G;L^1,L^1;P]$ where $P=\big(\hat{\mathsf f}(uv)\big)_{u,v\in L^1}$.
\begin{enumerate}
\item\label{item:1} The operation $\cdot$ is associative (and we denote  by $\mathcal{Z}[G,L,{\mathsf
f}]$ the semigroup $Z$).
\item\label{item:2} The mapping $\check{\mathsf f}:X^+\rightarrow Z$ is an epimorphism.
\item\label{item:3} For every $y,w\in X^+$, $\check{\mathsf f}(y)=\check{\mathsf f}(w)$ if and only if $y=w$ in
$T$.
\item\label{item:4} The semigroup $Z$ is isomorphic to $T$.

\item\label{item:5} $Z$ is a local group with minimal ideal $I=L^1GL^1$ isomorphic to $J$.
\end{enumerate}
\end{proposition}
\begin{proof} In order to verify~\ref{item:1}, let $z_1,z_2,z_3\in Z$ and notice that $z_1\cdot (z_2\cdot z_3)=\check{\mathsf f}(z_1\check{\mathsf
f}(z_2z_3))$ by definition of $\cdot$. Now, as  $z_1\in Z$, we have
$z_1=\check{\mathsf f}(z_1)$ and so,
by~\eqref{eq:property_check_function2},  $z_1\cdot (z_2\cdot
z_3)=\check{\mathsf f}(z_1z_2z_3)$. By symmetry $(z_1\cdot z_2)\cdot
z_3 =\check{\mathsf f}(z_1z_2z_3)$ which shows the associativity of
$\cdot$.

Let $y,w\in X^+$. Then $\check{\mathsf f}(y)$ and $\check{\mathsf
f}(w)$ belong to $Z$ and, so, by definition of $\cdot$ and
by~\eqref{eq:property_check_function2}, $\check{\mathsf f}(y)\cdot
\check{\mathsf f}(w)=\check{\mathsf f}(\check{\mathsf
f}(y)\check{\mathsf f}(w))=\check{\mathsf f}(yw)$. Since
$\check{\mathsf f}$ is clearly onto,~\ref{item:2} is proved.

For~\ref{item:3}, suppose first that $\check{\mathsf f}(y)=\check{\mathsf f}(w)$. By
Lemma~\ref{lemma:w_evaluated_in_T}, $y=\check{\mathsf f}(y)$ and $w=\check{\mathsf f}(w)$ in $T$, whence also
$y=w$ in $T$. Suppose next that $y=w$ in $T$. Without loss of generality we may assume that $w$ is deduced from
$y$ in one step, so that $y=y'x_1y''$ and $w=y'x_2y''$ with $(x_1=x_2)\in R$. Since by~\ref{item:2}
$\check{\mathsf f}$ is a homomorphism, to deduce $\check{\mathsf f}(y)=\check{\mathsf f}(w)$ it suffices to show
that $\check{\mathsf f}(x_1)=\check{\mathsf f}(x_2)$. If $(x_1=x_2)\in R_G$ then $x_1,x_2\in A_G^+$ and
$x_1=x_2$ in $G$. Therefore, if $x\in G$ represents both $x_1$ and $x_2$, then $\check{\mathsf
f}(x_1)=x=\check{\mathsf f}(x_2)$. Suppose next that $x_1=x_2$ is the relation $r_u\in R_{\mathsf f}$, that is
$eue={\mathsf f}(u)$, for some $u\in L$. We have $\check{\mathsf f}(eue)=e\hat{\mathsf f}(u)e$ by definition of
$\check{\mathsf f}$. As $u\in L$, $\hat{\mathsf f}(u)={\mathsf f}(u)$. Since $e$ represents the identity of $G$,
$e{\mathsf f}(u)e={\mathsf f}(u)$ in $G$ and so $\check{\mathsf f}(eue)={\mathsf f}(u)=\check{\mathsf
f}({\mathsf f}(u))$. It remains to treat the case where $x_1=x_2$ is the relation $r_{\ddot v}$, that is $\ddot
v=\ddot v_{\alpha}{\mathsf f}(\ddot v)\ddot v_{\omega}$, of $R_{\mathsf f}$ for some $\ddot v\in \ddot{L}$. If
$\ddot v\in \ddot{L}$ then $\mbox{sc}_L[\ddot v]=(\ddot v_{\alpha},\ddot v,\ddot v_{\omega})$ and so
$\check{\mathsf f}(\ddot v)=\ddot v_{\alpha}{\mathsf f}(\ddot v)\ddot v_{\omega}$. Now,  $\ddot
v_{\alpha}{\mathsf f}(\ddot v)\ddot v_{\omega}=\check{\mathsf f}(\ddot v_{\alpha}{\mathsf f}(\ddot v)\ddot
v_{\omega})$ because $\ddot v_{\alpha}{\mathsf f}(\ddot v)\ddot v_{\omega}\in Z$, whence $\check{\mathsf
f}(\ddot v)=\check{\mathsf f}(\ddot v_{\alpha}{\mathsf f}(\ddot v)\ddot v_{\omega})$. This completes the proof
of~\ref{item:3}.

To deduce~\ref{item:4} it suffices to  notice that the existence of an isomorphism $\theta:T\rightarrow Z$ is an
immediate consequence of~\ref{item:2} and ~\ref{item:3}. It is the unique mapping from $T$ onto $Z$ such that
$\theta\circ \psi=\check{\mathsf f}$, where $\psi$  is the canonical epimorphism from $X^+$ onto $T$.

It is not difficult to verify that $I=L^1GL^1$ is the minimal ideal of $Z$ and that $\varphi:I\rightarrow J$,
$ugv\mapsto (u,g,v)$, defines an isomorphism from $I$ onto $J$. Moreover, the Rees quotient $Z/I=L\cup\{0\}$ is
a nilpotent semigroup (i.e., $Z/I$ has $0$ as its unique idempotent). Hence, by
Proposition~\ref{prop:caracterization_local_group}, $Z$ is a local group and this finishes the proof of the
proposition.
\end{proof}

Let $S=L\cup J=L\cup (L^1\times G\times L^1)$. We extend the mapping
$\varphi:I\rightarrow J$ above to a new one  $\varphi:Z\rightarrow
S$ by setting $\varphi(w)=w$ for every $w\in L$. Next, we define an
operation $\odot$ in $S$ by setting, for every $w,w'\in L$ and
$(u,g,v),(u',g',v')\in J$,
$$\begin{array}{rlrl}
w\odot w'&\hspace*{-3mm}= \varphi(w\cdot w'),&
w\odot(u,g,v)&\hspace*{-3mm}=\varphi(w\cdot ugv), \\
(u,g,v)\odot w&\hspace*{-3mm}=\varphi(ugv\cdot
w),\hspace*{5mm}&
(u,g,v)\odot(u',g',v')&\hspace*{-3mm}=\varphi(ugv\cdot u'g'v').
\end{array}$$
As one may verify, this operation makes $S$ a semigroup, which we
denote by $\mathcal{S}[G,L,{\mathsf f}]$, and $\varphi$ an
isomorphism from $Z$ onto $S$. The following statement is then an
immediate consequence of previous results.
\begin{corollary}\label{corol:semigroups_S_Z_T} Let  $L\subseteq A^+$ be a non-empty factorial language,
 let $G$ be a group defined by a presentation $\langle A_G\mid
R_G\rangle$  and let ${\mathsf f}:L\cup \ddot{L}\rightarrow G$ be a
function.  The semigroups $\mathcal{S}[G,L,{\mathsf f}]$,
$\mathcal{T}[G,L,{\mathsf f}]$ and $\mathcal{Z}[G,L,{\mathsf f}]$
are isomorphic local groups, defined by the presentation $\langle
A\cup A_G\mid R_G\cup R_{\mathsf f}\rangle$, with minimal ideal
$\mathcal{M}[G;L^1,L^1;P]$, where $P=\big(\hat{\mathsf
f}(uv)\big)_{u,v\in L^1}$, and $L$ as set of non-regular elements.
Therefore, the above semigroups are finite if and only if both $L$
and $G$ are finite.
\end{corollary}

\section{Generators for  $\LH$}\label{section:cis}
In this section, we let $\He$ denote a pseudovariety of groups and show that the semigroups
$\mathcal{S}[G,L,{\mathsf f}]$, with $G\in \He$ and $L$ finite, form a generating set of the pseudovariety
$\LH$. For this we will use the results of~\cite{Almeida&Azevedo:1993}, where Almeida and Azevedo study
semidirect products of the form $\V*\D_k$, with $\V$ a pseudovariety of semigroups which is not locally trivial
and $\D_k$
 the pseudovariety of all finite semigroups that verify the
identity $yx_1\cdots x_k=x_1\cdots x_k$. Recall that $\bigcup_{k\geq 1}{\bf D}_k=\D$. We are interested only in
the cases where $\V$ is $\He$, since $\LH={\bf H*D}=\bigcup_{k\geq 1}{\bf H\ast D}_k$, and so we only recall the
corresponding results.

 Denote by $L_k=A^{\leq k}$ the
(factorial) language of all words over $A$ of length at most $k$ and notice that $\ddot L_k=A^{k+1}$. For a word
$u\in A^+$, let ${\mathtt i}_k(u)$ (resp.\ ${\mathtt t}_k(u)$) be the longest prefix (resp.\ suffix) of $u$ of
length at most $k$. A function ${\hbar}:A^+\rightarrow G$ into a group $G$ is said to be a $k$-superposition
homomorphism if $\hbar(uv)=\hbar(u){\hbar}({\mathtt t}_k(u)v)={\hbar}(u {\mathtt
i}_k(v)){\hbar}(v)$ and ${\hbar}(w)=1_G$ for every $u,v\in A^+$ and $w\in L_k$. Notice that, for a word
$w\in A^+\setminus L_k$, ${\hbar}(w)={\hbar}(w_1){\hbar}(w_2)\cdots {\hbar}(w_{|w|-k})$ where
$w_1,w_2,\ldots,w_{|w|-k}\in\ddot L_k$ are the successive factors of length $k+1$ of $w$. Therefore a
$k$-superposition homomorphism ${\hbar}:A^+\rightarrow G$ is defined by  ${\hbar}(\ddot L_k)$.

Given a $k$-superposition homomorphism ${\hbar}:A^+\rightarrow
G$ into a group $G\in\He$, let
$$M_k(G,{\hbar})=\{(v,1_G,v)\mid v\in L_{k-1}\}\cup (A^k\times G\times A^k)$$ be the
semigroup with multiplication given, for every $u,v,u',v'\in L_k$ and $g,g'\in G$, by
$$(u,g,v)(u',g',v')=({\mathtt i}_k(uu'),g{\hbar}(vu')g',{\mathtt t}_k(vv')).$$ Note that $I=A^k\times G\times
A^k$ is the minimal ideal of $M_k(G,{\hbar})$, precisely the
Rees matrix semigroup $\mathcal{M}[G;A^k,A^k;P]$ with
$P=\big({\hbar}(uv)\big)_{u,v\in A^k}$, and that all the
elements of $\{(v,1_G,v)\mid v\in L_{k-1}\}$ are non-regular, whence
$M_k(G,{\hbar})\in\LH$. Furthermore, we have the following
specialization of~\cite[Corollary 10.6.8]{Almeida:1992}.
\begin{proposition}\label{prop:generators_Mk} The pseudovariety $\LH$
is generated by the semigroups of the form $M_k(G,{\hbar})$ with
$k\geq 1$,  $G\in\He$, $A$ an alphabet and ${\hbar}:A^+\rightarrow G$ a $k$-superposition homomorphism.
\end{proposition}

Given a  semigroup  $M_k(G,{\hbar})$, we now identify a finite semigroup $\mathcal{S}[G,L,{\mathsf f}]$,
 also denoted  $S_k(G,{\mathsf f})$, such that $M_k(G,{\hbar})$ is a subsemigroup of a homomorphic image of $S_k(G,{\mathsf f})$. The group
$G$ is of course the same for the two semigroups. We take for $L$
the language $L_k=A^{\leq k}$. This way the minimal ideal
$I=A^k\times G\times A^k$ of $M_k(G,{\hbar})$ is a subset of the
minimal ideal $J=L_k^1\times G\times L_k^1$ of $S_k(G,{\mathsf f})$.
Finally, notice that ${\hbar}$ has domain $A^+$ whereas we need
to define a function ${\mathsf f}$ with domain $L_k\cup \ddot
L_k(=L_{k+1})$ and we want to make $I$ a subsemigroup of $J$. But
this is no problem since ${\hbar}$ is a $k$-superposition
homomorphism and so ${\hbar} (L_k)=\{1_G\}$ and the  value of
${\hbar}$ on $A^+\setminus L_k$ is determined by its value on
$\ddot L_k$. We let ${\mathsf f}$ be the restriction of ${\hbar}$ to $L_k\cup \ddot L_k$. Therefore, with the above choices,
${\hbar}$ is precisely the restriction of $\hat{\mathsf f}$ to
$A^+$. Since the function $\hat{\mathsf f}$ is the one that
determines the structure matrix of the minimal ideal $J$ of
$S_k(G,{\mathsf f})$, the objective of turning $I$ into a
subsemigroup of $J$ is guaranteed. The difference between
$M_k(G,{\hbar})$ and $S_k(G,{\mathsf f})$ is not only in their
minimal ideals: while the set of non-regular elements of
$S_k(G,{\mathsf f})$ is $L_k$, the set of non-regular elements of
$M_k(G,{\hbar})$ is in bijection with $L_{k-1}$. That is, each
word $u\in A^k$ is an element of $S_k(G,{\mathsf f})$ that lies
${\cal J}$-above the minimal ideal while in $M_k(G,{\hbar})$ it
corresponds to the element $(u,1_G,u)$ that lies in the minimal
ideal. On the other hand $S_k(G,{\mathsf f})$ also has an element
$(u,1_G,u)$. So, we construct a new semigroup
 in which the elements $u$ and $(u,1_G,u)$ are identified (and only these ones).

\begin{lemma}\label{lemma:S'k}
With the above choices for $L$ and ${\mathsf f}$, let
$S'_k(G,{\mathsf f})$ be a semigroup defined by the presentation
$\langle X\mid R'\rangle$, where $R'=R_G\cup R_{\mathsf
f}\cup\{u=ueu:u\in A^k \}$.
\begin{enumerate}
\item\label{item:1} The semigroup $S'_k(G,{\mathsf f})$
is a homomorphic image of $S_k(G,{\mathsf f})$.
\item\label{item:2} $\rho_{R'}=\rho_R\cup\{u=ueu:u\in A^k \}$, where $R=R_G\cup R_{\mathsf
f}$.
\end{enumerate}
\end{lemma}
\begin{proof}
To deduce~\ref{item:1} it suffices to note that $S'_k(G,{\mathsf
f})$ is defined by a presentation that differs from the one defining
$S_k(G,{\mathsf f})$ only in having the extra relations $u=ueu$
($u\in A^k$).

The inclusion $\rho_R\cup\{u=ueu:u\in A^k \}\subseteq \rho_{R'}$ is immediate from the definition of a
congruence generated by a relation. To prove the reverse inclusion, let $w,z\in X^+$ and suppose that $(w=z)\in
\rho_{R'}$ and $(w=z)\not\in \{u=ueu:u\in A^k \}$. We need to prove that $(w=z)\in \rho_{R}$. This holds
trivially if the relation $w=z$ is deduced  without using the relations $u=ueu$ ($u\in A^k$). We assume, without
loss of generality, that $w=w'uw''$ and $z=w'ueuw''$ with $u\in A^k$ and $w'$ or $w''$ non-empty. We treat only
the case in which $w'$ is non-empty since the case $w''$ non-empty is similar. If ${\mathtt t}_1(w')\in A_G$
then, by~\eqref{eq:e_neutral element} and using relation $r_u$, one deduces that $z=w'eueuw''=w'euw''=w$ in
$X^+/\rho_R$. If ${\mathtt t}_1(w')=a\in A$ then $w'=w'''a$ for some $w'''\in X^*$ and $au\in A^{k+1}=\ddot
L_k$. Notice that $(au)_\alpha={\mathtt i}_k(au)$ and $(au)_\omega={\mathtt t}_k(au)=u$. Hence, using relation
$r_{au}$ one deduces that $z=w'''{\mathtt i}_k(au){\mathsf f}(au)ueuw''$ is in $\rho_R$. Since ${\mathsf
f}(au)\in A_G^+$ we may now proceed as in the previous case. So, using relations $r_{u}$ and $r_{au}$, we deduce
that $z=w'''{\mathtt i}_k(au){\mathsf f}(au)uw''=w'''auw''=w$ in $X^+/\rho_R$. Therefore, in both cases,
$(w=z)\in\rho_R$. This proves~\ref{item:2}.
\end{proof}

Note that, informally speaking, Lemma~\ref{lemma:S'k}~\ref{item:2}
states that the relation $u=ueu$ ($u\in A^k$) makes no
identifications on $S_k(G,{\mathsf f})$ other than the elements $u$
and $(u,1_G,u)$. It should now be clear that the mapping
$\phi:M_k(G,{\hbar})\rightarrow S'_k(G,{\mathsf f})$ given, for
$v\in L_{k-1}$, $u_1,u_2\in A^k$ and $g\in G$, by $\phi(v,1,v)=v$
and $\phi (u_1,g,u_2)=(u_1,g,u_2)$ is a monomorphism from
$M_k(G,{\hbar})$ into $ S'_k(G,{\mathsf f})$. The following
result is then an immediate consequence of
Proposition~\ref{prop:generators_Mk} and Lemma~\ref{lemma:S'k}.

\begin{corollary}\label{corol:generators_Sk} The pseudovariety $\LH$
is generated by the semigroups of the form $S_k(G,{\mathsf f})$ with
$k\geq 1$,  $G\in\He$, $A$ an alphabet and ${\mathsf
f}:L_{k+1}\rightarrow G$ with ${\mathsf f}(L_{k})=\{1_G\}$.
\end{corollary}

We observe that, in general, the semigroup $M_k(G,{\hbar})$ is
not a homomorphic image of $S_k(G,{\mathsf f})$. That is, it is not
possible to obtain a presentation for $M_k(G,{\hbar})$ simply by
adding new relations to the presentation that defines
$S_k(G,{\mathsf f})$. For instance, consider the alphabet $A=\{a\}$,
the cyclic group $G=\langle g\mid g^3=g\rangle=\{g,e\}$ of order 2,
$k=1$ and ${\hbar}(a^2)=g$. Then $M_k(G,{\hbar})=\{a\}\times
G\times \{a\}\cong G$ and $S_k(G,{\mathsf f})=\langle a,g\mid
g^3=g,eae=e,a^2=aga\rangle=\{a\}\cup (\{1,a\}\times G\times
\{1,a\})$. In order to obtain $M_k(G,{\hbar})$ as a homomorphic
image of $S_k(G,{\mathsf f})$ one should identify the idempotents
$(1,e,1)$ and $(a,g,a)$ of $S_k(G,{\mathsf f})$, by adding the
relation $e=aga$ to the above presentation. But then we would have
$a^2=e$ and so $e=ee=a^2a^2=aa^2a=aea=a^2ga^2=ege=g$ in the
resulting semigroup. This semigroup would therefore be aperiodic and
so different from $M_k(G,{\hbar})$.

\section{Canonical forms for $\kappa$-terms of rank 1}\label{section:kappa-terms_rank1}
A $\kappa$-term is a formal expression obtained from letters of an
alphabet $A$ using two operations: the binary concatenation and the
unary $(\omega-1)$-power. The {\em rank} of a $\kappa$-term is the maximum number of nested $(\omega-1)$-powers in it. A $\kappa$-term has a natural
interpretation on each finite semigroup $S$: the concatenation is
viewed as the semigroup multiplication while the $(\omega-1)$-power
is interpreted as the unary operation which sends each element $s$
of S to the inverse of $s^{\omega+1}(=ss^\omega)$ in the maximal
subgroup containing the unique idempotent power $s^\omega$ of $s$.
For a class $\mathcal{C}$ of finite semigroups and $\kappa$-terms
$\pi_1$ and $\pi_2$, we say that $\mathcal{C}$ satisfies the
$\kappa$-identity $\pi_1=\pi_2$, and write
$\mathcal{C}\models\pi_1=\pi_2$, if $\pi_1$ and $\pi_2$ have the
same interpretation over every semigroup of~$\mathcal{C}$. The
$\kappa$-word problem for $\mathcal{C}$ consists in deciding, given
a $\kappa$-identity $\pi_1=\pi_2$, whether
$\mathcal{C}\models\pi_1=\pi_2$. A solution for this problem has
been obtained for some important pseudovarieties and we will soon
present solutions for the pseudovarieties $\Se$~\cite{Costa:2013}
and $\LG$~\cite{Costa&Nogueira&Teixeira:2013}. In the current paper
we will treat an instance of the problem by showing that $\Se$ and
$\LG$ satisfy the same $\kappa$-identities $\pi=\rho$ where $\pi$
and $\rho$ have rank at most 1, and that it is decidable whether
 $\Se$ and $\LG$ satisfy $\pi=\rho$. For that, we first reduce each $\kappa$-term of rank 1 to a certain
canonical form, similar to the normal form introduced by McCammond in~\cite{mccammond2001} to solve the
$\omega$-word problem for the pseudovariety $\A$ of finite aperiodic semigroups. The definition of an
$\omega$-term differs from that of a $\kappa$-term only in the use of the $\omega$-power instead of the
$(\omega-1)$-power (we remark that the two operations coincide on the pseudovariety $\A$). Then we use our
test-semigroups ${\cal S}(G,L,{\mathsf f})$ to separate two distinct canonical forms.

A word is said to be {\em primitive} if it cannot be written in the form $u^n$ with $n>1$. We say that two words
$w$ and $z$ are {\em conjugate} if there exist words $u,v\in A^*$ such that $w=uv$ and $z=vu$.  Let an order be
fixed for the letters of the alphabet $A$. A {\em Lyndon word} is a primitive word which is minimal, with
respect to the lexicographic ordering, in its conjugacy class. For combinatorial properties involving Lyndon
words that are relevant for the remaining of this paper, the reader is referred
to~\cite{Almeida&Costa&Zeitoun:2012}, where an alternative proof of correctness of McCammond's normal form
algorithm over $\A$ is presented.

We employ the following notation for $\kappa$-terms, where $n>0$: $x^{\omega+n}$  represents $x^{\omega}x^n$;
$x^{\omega+0}$ is $x^{\omega}$; $x^{\omega-n}$ denotes $(x^{\omega-1})^{n}$. The $\kappa$-terms of rank $0$ are
the words from $A^+$ and they are all considered to be in canonical form. A $\kappa$-term of rank $1$ is an
expression $\pi$ of the form
\begin{equation}\label{eq:normal:form}
\pi=u_0x_1^{\omega+q_1}u_1x_2^{\omega+q_2} \cdots x_m^{\omega+q_m} u_m
\end{equation}
with $m\geq 1$, $u_0,\ldots,u_m\in A^*$, $x_1,\ldots,x_m\in A^+$ and $q_1,\ldots,q_m\in \Z$. Using the
terminology of McCammond~\cite{mccammond2001}, each factor of the form $x_i^{\omega+q_i}
u_ix_{i+1}^{\omega+q_{i+1}}$ will be called a \emph{crucial portion of $\pi$}. The prefix $u_0x_1^{\omega+q_1}$
and the suffix $x_m^{\omega+q_m} u_m$ will be called respectively the \emph{initial portion} and the \emph{final
portion} of $\pi$.

\begin{definition}[Canonical form in rank 1]\label{def:normalform1}
 The \kt\ $\pi$ in~\eqref{eq:normal:form} is in {\em canonical form}
if
\begin{enumerate}
  \item\label{item:1:normalform1}  each $x_i$ is a Lyndon word;

  \item\label{item:2:normalform1} $x_i$ is not a suffix of $u_{i-1}$ for $1\leq i\leq m$;

  \item\label{item:3:normalform1}  $x_i$ is not a prefix of $u_{i}x_{i+1}^{|x_i|}$ for $1\leq i\leq m$, where $x_{m+1}$ is the empty word.
\end{enumerate}
\end{definition}

For instance, let $a,b\in A$ be letters such that $a<b$. The $\kappa$-terms $b(abb)^{\omega-1} b^{\omega+2}$ and
$a^{\omega+3} b(aab)^{\omega}aa(ab)^{\omega-2}$ are in canonical form, as is also any $\kappa$-term of the type
$(ax^n)^{\omega+p} x^{\omega+q}$ or $y^{\omega+p} b(y^nb)^{\omega+q}$ where $x\in\{a,b\}^*b\{a,b\}^*$ and
$y\in\{a,b\}^*a\{a,b\}^*$ are Lyndon words and $n,p,q\in\Z$ with $n\geq 1$. On the contrary, the $\kappa$-terms $a(bab)^{\omega-1} (aa)^{\omega}$ and
$(ab)^{\omega} abab^{\omega-1}b^3$ are not in canonical form. We remark that our condition to the
canonical form of each crucial portion is different from the one that McCammond~\cite{mccammond2001} imposed to
crucial portions of $\omega$-terms. That is, if $x_i^{\omega+q_i} u_ix_{i+1}^{\omega+q_{i+1}}$ is in canonical
form, then  the $\omega$-term $x_i^{\omega} u_ix_{i+1}^{\omega}$ may not be in McCammond's normal form. To each
crucial portion $x_i^{\omega+q_i} u_ix_{i+1}^{\omega+q_{i+1}}$ there is an associated bi-infinite word $\cdots
x_ix_iu_ix_{i+1}x_{i+1}\cdots$, which we  denote by $\infee{x}{i}u_i\infdd{x}{i+1}$. From the definition of
canonical form, it is easy to deduce that, for $i,j\in\{1,\ldots,m-1\}$, the bi-infinite words
$\infee{x}{i}u_i\infdd{x}{i+1}$ and $\infee{x}{j}u_{j}\infdd{x}{j+1}$ coincide if and only if $x_i=x_j$,
$u_i=u_j$ and $x_{i+1}=x_{j+1}$.

As we shall see below, for each $\kappa$-term $\alpha$ of rank $1$ there is a $\kappa$-term $\alpha'$ in canonical form such
that $\Se\models\alpha=\alpha'$. The $\kappa$-term $\alpha'$ is unique by Theorem~\ref{theo:wp_rank1} below and
so we call it \emph{the canonical form of $\alpha$}. Moreover it can be computed from $\alpha$ by applying
elementary changes resulting from reading in either direction the following $\kappa$-identities, where
$i,j,n\in\Z$ with $n>0$,
\begin{xalignat*}{3}
&1.~(x^{n})^{\omega+j}=x^{\omega+nj}&&3_R.~x^{\omega+i}x=x^{\omega+i+1}
&&4.\ \;\!~(xy)^{\omega+i}x=x(yx)^{\omega+i}      \\[-1mm]
&2.~x^{\omega+i} x^{\omega+j}=x^{\omega+i+j}
&&3_L.~xx^{\omega+i}=x^{\omega+i+1}         &&
\end{xalignat*}
These $\kappa$-identities are easily shown to be valid in all finite
semigroups and so $\alpha=\alpha'$ is indeed verified by $\Se$. If a
subterm given by the left side of a $\kappa$-identity of type 1–-3
is replaced in a $\kappa$-term by the right side of the
$\kappa$-identity, then we say there is a \emph{contraction} of that
type. If the replacement is done in the opposite direction then we
say that there is an \emph{expansion} of that type. An application
of the $\kappa$-identity 4, in either direction, will be called a
\emph{shift}. For example, consider the $\kappa$-term
$\alpha=(bababa)^\omega b^{\omega-3} b(bb)^{\omega+1}$. The
following sequence of $\kappa$-terms, starting in $\alpha$, is
derived from the $\kappa$-identities 1–-4 above,
$$\begin{array}{rl}
\alpha=&\hspace*{-2mm}(ba)^\omega b^{\omega-3} b(bb)^{\omega+1}
=(ba)^\omega b^{\omega-3} bb^{\omega+2}
=(ba)^{\omega-1} ba b^{\omega-3} bb^{\omega+2}\\
=&\hspace*{-2mm} b(ab)^{\omega-1} a b^{\omega-3}
bb^{\omega+2}=b(ab)^{\omega-1} ab^{\omega-2}
b^{\omega+2}=b(ab)^{\omega-1} ab^{\omega}\\
=&\hspace*{-2mm}b(ab)^{\omega-1} abb^{\omega-1}=b(ab)^{\omega}
b^{\omega-1}.
\end{array}$$
 The two first steps in this derivation are
contractions of type $1$, the third step is an expansion of type $3_R$, the fourth is a shift,  the fifth is a
contraction of type $3_R$, the sixth step is a contraction of type $2$, the seventh is an expansion of type
$3_L$ and the final step is a contraction of type $3_R$. As $b(ab)^{\omega} b^{\omega-1}$ is in canonical form,
we conclude that this $\kappa$-term is the canonical form of $\alpha$. The steps of the algorithm to compute the
canonical form of an arbitrary $\kappa$-term of rank $1$ may be described as follows.

\begin{enumerate}[label=(\arabic*)]
\item Apply all possible contractions of type~1.

\item By means of an expansion of type $3$ and a shift,
  write each infinite power in the form $x^{\omega+q}$ where $x$ is a Lyndon word.

\item\label{item:contrations_type3} Apply all possible contractions of type~3.

\item Apply all possible contractions of type~2.

  \item Standardize each crucial portion $x^{\omega+p}uy^{\omega+q}$ as follows. By step~\ref{item:contrations_type3}, $x$
  is not a prefix and $y$ is not a suffix of $u$. Let $\ell$ be the minimum non negative integer such that
  $|uy^\ell|\geq |x|$. If $x$ is not a prefix of $uy^\ell$ then the crucial portion
  $x^{\omega+p}uy^{\omega+q}$ is already in canonical form. Otherwise $\ell\neq 0$. In this case, apply $\ell$ expansions of type
  $3_L$  to the limit term on the right side of the crucial portion, followed by all $n\geq 1$ possible contractions of type $3_R$. As shown in~\cite{Costa:2013}, the crucial portion $x^{\omega+p+n}v
  y^{\omega+q-\ell}$ thus obtained is already in canonical form.
\end{enumerate}

It is not difficult to check that the above procedure does indeed transform any $\kappa$-term of rank 1 into one
in canonical form.

\section{Identities involving $\kappa$-terms of rank at most 1}\label{section:kappa-terms_rank1}
The main objective of this section is to prove that $\Se$ and $\LG$
satisfy the same $\kappa$-identities involving $\kappa$-terms of
rank at most 1. This property cannot, obviously, be extended to rank
2 since $\Se$ does not verify the $\kappa$-identity $(x^\omega
yx^\omega)^\omega=x^\omega$ while $\LG$ verifies it (in fact this
$\kappa$-identity  defines $\LG$).

Consider rank 1  $\kappa$-terms $\pi=u'_0x_1^{\omega+q_1}u_1 \cdots x_m^{\omega+q_m}u_m$ and
$\rho=u'_mx_{m+1}^{\omega+q_{m+1}}u_{m+1} \cdots x_n^{\omega+q_n}u_n$ in canonical form. We associate to the
$\kappa$-identity $\pi=\rho$ an identity  ${\mathsf w}_\pi={\mathsf w}_\rho$ and a finite local group
$S_{\pi,\rho}$, of the form $\mathcal{S}(G,L,{\mathsf f})$, as follows. As we shall see, the identity ${\mathsf
w}_\pi={\mathsf w}_\rho$ (together with another condition easy to verify) serves to decide whether the
$\kappa$-identity $\pi=\rho$ holds over $\LG$ and $S_{\pi,\rho}$ is a test-semigroup for $\pi=\rho$. We set
$I_n=\{1,\ldots,n\}$ and $I_{m,n}=I_n\setminus\{ m,n\}$. We also set, for each $i\in I_n$, $\ell_i=|x_i|$ and
$x_i=a_{i1}a_{i2}\cdots a_{i\ell_i}$ with $a_{i1},a_{i2},\ldots, a_{i\ell_i}\in A$. For each
$j\in\{1,\ldots,\ell_i\}$, we let $x_{ij}$ be the conjugate $x_{ij}=a_{ij}\cdots a_{i\ell_i}a_{i1}\cdots
a_{i\;\!j-1}$ of $x_i$ and notice that $x_{i1}=x_i$. Now, let $\mathbb{q}=1+\mbox{max}\{|q_i|:i\in I_n\}$ and,
for each $i\in I_n$, consider the positive integer $\mathbb{q}_{i}=\mathbb{q}+q_{i}$. We associate a variable
$\mathsf{v}_x$ to each word $x\in A^+$ and
 a variable $\mathsf{v}_{x,u,y}$ to each triple  $(x,u,y)$ of words
 of $A^+$.  Let $\mathsf{V}=\{\mathsf{v}_{x_i},\mathsf{v}_{x_j,u_j,x_{j+1}}:i\in I_n,j\in
I_{m,n}\}$. We associate to $\pi$ and $\rho$ the following words
$\mathsf{w}_\pi,\mathsf{w}_\rho\in \mathsf{V}^+$,
$$\begin{array}{rl}
 \mathsf{w}_\pi&\hspace*{-3mm}=\mathsf{u}'_{0}\mathsf{v}_{x_1}^{\mathbb{q}_1}\mathsf{v\hspace*{-1pt}}_{x_1,u_1,x_2}
\mathsf{v}_{x_2}^{\mathbb{q}_2}\cdots \mathsf{v}_{x_{m-1},
u_{m-1},x_m}\mathsf{v}_{x_m}^{\mathbb{q}_m}\mathsf{u}_{m},\\[1mm]
\mathsf{w}_\rho&\hspace*{-3mm}=\mathsf{u}'_{m}\mathsf{v}_{x_{m+1}}^{\mathbb{q}_{m+1}}
\mathsf{v}_{x_{m+1},u_{m+1},x_{m+2}}
\mathsf{v}_{x_{m+2}}^{\mathbb{q}_{m+2}}\cdots \mathsf{v}_{x_{n-1},
u_{n-1},x_{n}}\mathsf{v}_{x_{n}}^{\mathbb{q}_n}\mathsf{u}_{n},
\end{array}$$ where $\mathsf{u}'_{0}$ (resp.\ $\mathsf{u}_{m}$,
$\mathsf{u}'_{m}$, $\mathsf{u}_{n}$) is the variable
$\mathsf{v}_{x_1}$ (resp.\ $\mathsf{v}_{x_m}$,
$\mathsf{v}_{x_{m+1}}$, $\mathsf{v}_{x_n}$) if $u'_{0}$ (resp.\
$u_{m}$, $u'_{m}$, $u_{n}$) is non-empty and is the empty word
otherwise. This completes the definition of the identity
$\mathsf{w}_\pi=\mathsf{w}_\rho$.

Let us now define the test-semigroup
$S_{\pi,\rho}=\mathcal{S}(G,L,{\mathsf f})$. Denote by
$F_{\mathsf V}$ the free group generated by the alphabet ${\mathsf
V}$. We select first a group
homomorphism $\eta:F_{\mathsf V}\rightarrow G$ into a finite group
$G$, associating to each variable $\mathsf{v}_{*}$ an element
$\eta(\mathsf{v}_{*})=g_{*}$, in such a way that: if the identity
$\mathsf{w}_\pi=\mathsf{w}_\rho$ is non-trivial, then
$g_\pi=\eta({\mathsf w}_\pi)$ and $g_\rho=\eta({\mathsf w}_\rho)$
are distinct elements of $G$. We observe the following
\begin{claim}
 We may assume that each element of $\eta(\mathsf{V})$  has an
order at least $2\mathbb{q}$.
\end{claim}
\begin{proof} In order to attest the claim, let $C_{2\mathbb{q}}=\langle s\rangle$
denote the cyclic group of order $2\mathbb{q}$ generated by an
element $s$ and consider the group $G'=G\times C_{2\mathbb{q}}$. The
elements $g'_{*}=(g_{*},s)$ of $G'$ have an order at least
$2\mathbb{q}$. Let $\eta':F_{\mathsf V}\rightarrow G'$ be the group
homomorphism defined by $\eta'(\mathsf{v}_{*})=g'_{*}$. The first
components of $g'_\pi=\eta'({\mathsf w}_\pi)$ and
$g'_\rho=\eta'({\mathsf w}_\rho)$ are respectively $g_\pi$ and
$g_\rho$. If these are distinct elements of $G$, then $g'_\pi$ and
$g'_\rho$ are distinct elements of $G'$. Therefore, if necessary, we
replace $\eta$  by $\eta'$  thus proving the claim.
\end{proof}

We describe next the language $L$. We will fix three positive integers $\mathbb{i}<\mathbb{j}<\mathbb{k}$ and
define sets of words $W_\mathbb{i}$, $W_\mathbb{j}$ and $W_\mathbb{k}$ as follows. We choose first an integer
$\mathbb{i}>|u'_0x_1u_1\cdots x_mu_m|+|u'_mx_{m+1}u_{m+1}\cdots x_nu_n|$ and let
$W_\mathbb{i}=\{u'_0x_1^\mathbb{i},u'_mx_{m+1}^\mathbb{i},x_{m}^\mathbb{i}u_m,x_{n}^\mathbb{i}u_n\}$. Suppose
now that $n>2$. Each crucial portion $x_i^{\omega+q_i}u_ix_{i+1}^{\omega+q_{i+1}}$ ($i\in I_{m,n}$) of $\pi$ or
$\rho$ determines a bi-infinite word $x_i^{-\infty} u_ix_{i+1}^{+\infty}$, that is not periodic (i.e., it is not
of the form $\cdots vvv\cdots$, the infinite repetition  for both sides of the same finite word $v$) by
definition of canonical form. Therefore, for every $i,i'\in I_{m,n}$ and $\mathbb{j}$ large enough, either
$x_i^{\mathbb{j}}
 u_ix_{i+1}^{\mathbb{j}}$ does not occur in $x_{i'}^{-\infty}
 u_{i'}x_{{i'}+1}^{+\infty}$, or, as a consequence of the rank 1 canonical form
 definition, it has exactly one occurrence in $x_{i'}^{-\infty}
 u_{i'}x_{{i'}+1}^{+\infty}$ (in which case $x_i=x_{i'}$, $u_i= u_{i'}$ and $x_{i+1}=x_{{i'}+1}$). We then say that $x_i^{\mathbb{j}}
 u_ix_{i+1}^{\mathbb{j}}$ \emph{is synchronized in} $x_{i'}^{-\infty}
 u_{i'}x_{{i'}+1}^{+\infty}$. We fix one such $\mathbb{j}$ with $\mathbb{j}>\mathbb{i}$, let
 $W_\mathbb{j}=\{x_i^\mathbb{j}u_ix_{i+1}^\mathbb{j}:i\in I_{m,n}\}$ and observe that $W_\mathbb{i}$ and $W_\mathbb{j}$ are disjoint sets. If $n=2$ then $\pi$ and $\rho$ have no crucial portions and we
 let $\mathbb{j}>\mathbb{i}$ and take for $W_\mathbb{j}$ the empty set. For any $n$, we fix at last an integer $\mathbb{k}>\mathbb{j}$  such
 that $x_{iq}^\mathbb{k}\not\in F(W_\mathbb{j})$ for every $i\in I_n$ and $q\in\{1,\ldots,\ell_i\}$. Since we may take for $\mathbb{k}$ any sufficiently large positive integer, we pick out an option
  such that $\mathbb{k}'=\mathbb{k}+1+\mathbb{q}$ is a multiple of the order of each element $g_{x_i}$ with $i\in I_n$.  Finally, we set $W_\mathbb{k}=\{x_{iq}^\mathbb{k}:i\in
 I_n, q\in\{1,\ldots,\ell_i\}\}$ and $L=F(W_\mathbb{i}\cup W_\mathbb{j}\cup
 W_\mathbb{k})$.

We conclude the definition of the semigroup $S_{\pi,\rho}$ by
identifying the mapping ${\mathsf f}:L\cup\ddot L\rightarrow G$. We
consider the subset
$W'_\mathbb{k}=\{x_1^\mathbb{k},\ldots,x_n^\mathbb{k}\}$ of
$W_\mathbb{k}$ and define
$${\mathsf f}(w)=\left\{\begin{array}{ll}
1_G&\mbox{if $w\in (L\cup\ddot L)\setminus (W_\mathbb{j}\cup W'_\mathbb{k})$}\\
g_{x_i}^{-1}g_{x_i,u_i,x_{i+1}}g_{x_{i+1}}^{-t_i-1}&\mbox{if
$w=x_i^\mathbb{j}
 u_ix_{i+1}^\mathbb{j}\in W_\mathbb{j}$}\\
 g_{x_i}&\mbox{if $w=x_i^\mathbb{k}\in W'_\mathbb{k}$}
\end{array}\right.$$
where $t_i\geq 0$ is the biggest integer such that  $x_{i+1}^{t_i}$
is a suffix of  $x_{i}^\mathbb{j} u_i$.

We may now prove the main result of this section.
\begin{theorem}\label{theo:wp_rank1}
Under the above assumptions and with the above notations, the
following conditions are equivalent:
\begin{enumerate}
\item\label{item:a} $\Se\models\pi=\rho$.

\item\label{item:b} $\LG\models\pi=\rho$.

\item\label{item:c} $S_{\pi,\rho}\models\pi=\rho$.

\item\label{item:d} ${\mathsf w}_\pi={\mathsf w}_\rho$, $u'_0=u'_m$ and
$u_m=u_n$.

\item\label{item:e} $\pi$ and $\rho$ are the same \kt.
\end{enumerate}
\end{theorem}
\begin{proof} The sequence of implications~\ref{item:d}$\Rightarrow$\ref{item:e}$\Rightarrow$\ref{item:a}$\Rightarrow$\ref{item:b}$\Rightarrow$\ref{item:c} holds
trivially. It remains therefore to prove the implication~\ref{item:c}$\Rightarrow$\ref{item:d}. Thus suppose
that $S_{\pi,\rho}\models\pi=\rho$ and let $\phi: T_A^\kappa\rightarrow S_{\pi,\rho}$ be the homomorphism of
$\kappa$-semigroups that coincides with $\varphi\circ \check{\mathsf f}$ on $A^+$, where $T_A^\kappa$ denotes
the $\kappa$-semigroup of all $\kappa$-terms and $\varphi$ is the isomorphism from $\mathcal{Z}[G,L,{\mathsf
f}]$ onto $S_{\pi,\rho}$ defined in Section~\ref{Presentations-local-groups}. Notice that, $\phi(w)=w$ when
$w\in L$ and, for $w\in A^+\setminus L$ with $\mbox{sc}_L[w]=(w_0,\ddot w_1,w_1,\ldots,\ddot w_p,w_p)$,
$\check{\mathsf f}(w)=w_0\bar{\mathsf f}(w)w_p$ and $\phi(w)=(w_0,\bar{\mathsf f}(w),w_p)$ where $\bar{\mathsf
f}(w)={\mathsf f}(\ddot w_1){\mathsf f}(w_1){\mathsf f}(\ddot w_2)\cdots {\mathsf f}(\ddot w_p)$. We begin by
computing the value under $\phi$ of each power $x_i^{\omega+q_i}$ with $i\in I_n$. By definition of
$W_\mathbb{k}$, $x_i^\mathbb{k}$ belongs to $L$, so that $\phi(x_i^\mathbb{k})=x_i^\mathbb{k}$. On the other
hand, by the assumption on $\mathbb{k}$, it is clear that $x_i^{\mathbb{k}+1}\notin L$, whence
$\phi(x_i^{\mathbb{k}+1})$ belongs to the minimal ideal of $S_{\pi,\rho}$. Therefore, the kernel $K_{x_i}$ of
the subsemigroup $\langle x_i\rangle$ of $S_{\pi,\rho}$ generated by $x_i(=\phi(x_i))$, is of the form
$$K_{x_i}=\{\phi(x_i^{\mathbb{k}+1}),\phi(x_i^{\mathbb{k}+2}),\ldots,
\phi(x_i^{\mathbb{k}+p_i})\}$$ with $p_i\geq 1$  minimal such that
$\phi(x_i^{\mathbb{k}+1+p_i})=\phi(x_i^{\mathbb{k}+1})$ and thus
\begin{equation}\label{eq:omega+qi}
\exists q'_i\in\{1,\ldots,p_i\},\
\phi(x_i^{\omega+q_i})=\phi(x_i^{\mathbb{k}+q'_i}).
\end{equation}
The positive integers $p_i$ and $q'_i$ are specified in the
following lemma.
\begin{lemma}\label{lemma:p_iandq'_i}
\begin{enumerate}
\item\label{item:a} $p_i$ is  the order of the element $g_{x_i}$ in $G$;

\item\label{item:b} $q'_i=\mathbb{q}_i+1$ (which is equivalent to saying that $\mathbb{k}+q'_i=\mathbb{k}'+q_i$).
\end{enumerate}
\end{lemma}
\begin{proof} Notice first that, for every $j\in\{1,\ldots,\ell_i\}$, the
word $v_j=x_{ij}^\mathbb{k} a_{ij}$ belongs to $\ddot{L}$,
$(v_j)_\alpha=x_{ij}^\mathbb{k}$,
$(v_j)_\omega=x_{i\;\!j+1}^\mathbb{k}$ if $j\neq\ell_i$ and
$(v_{\ell_i})_\omega=x_{i1}^\mathbb{k}$. Hence, one verifies that,
for $q\geq 1$,
$$\mbox{sc}_L[x_i^{\mathbb{k}+q}]=(x_{i1}^\mathbb{k},v_1,\ldots,x_{i\ell_i}^\mathbb{k},v_{\ell_i},
x_{i1}^\mathbb{k},v_1,\ldots,x_{i\ell_i}^\mathbb{k},v_{\ell_i},\ldots,
x_{i1}^\mathbb{k},v_1,\ldots,x_{i\ell_i}^\mathbb{k},v_{\ell_i},x_{i1}^\mathbb{k})$$
with $q$ repetitions of
$x_{i1}^\mathbb{k},v_1,\ldots,x_{i\ell_i}^\mathbb{k},v_{\ell_i}$. Therefore,
$$\bar{\mathsf f}(x_i^{\mathbb{k}+q})={\mathsf f}(x_{i1}^\mathbb{k})^{-1}\left({\mathsf f}(x_{i1}^\mathbb{k}){\mathsf f}(v_1)\cdots {\mathsf f}(x_{i\ell_i}^\mathbb{k}){\mathsf f}(v_{\ell_i})\right)^q.$$
On the other hand, ${\mathsf f}(x_{i1}^\mathbb{k})=g_{x_i}$ and ${\mathsf f}(v_j)=1_G$ for all $i$ and $j$, and ${\mathsf
f}(x_{ij}^\mathbb{k})=1_G$ for $j\neq 1$. It then follows that
\begin{equation}\label{eq:x_ikq}
\phi(x_i^{\mathbb{k}+q})=(x_{i}^\mathbb{k},\bar{\mathsf
f}(x_i^{\mathbb{k}+q}),x_{i}^\mathbb{k})=(x_{i}^\mathbb{k},g_{x_i}^{q-1},x_{i}^\mathbb{k}).
\end{equation}
We may now employ~\eqref{eq:x_ikq} to deduce that, for $p\geq 1$,
$\phi(x_i^{\mathbb{k}+1+p})=\phi(x_i^{\mathbb{k}+1})$ if and only if
$g_{x_i}^{p}=1_G$. Hence, the least positive integer $p_i$ such that
 $\phi(x_i^{\mathbb{k}+1+p_i})=\phi(x_i^{\mathbb{k}+1})$ is the order of $g_{x_i}$, thus proving~\ref{item:a}.

 For~\ref{item:b}, we show that $\mathbb{k}+q'_i=\mathbb{k}'+q_i$.  In view of
Claim 1, $p_{i}\geq 2\mathbb{q}$, whence
$\mathbb{k}'+q_i\in\{\mathbb{k}+1,\ldots,\mathbb{k}+ p_i\}$.
Therefore, to prove the equality $\mathbb{k}+q'_i=\mathbb{k}'+q_i$
it suffices to verify that
$\phi(x_i^{\omega})=\phi(x_i^{\mathbb{k}'})$. For this,
using~\eqref{eq:x_ikq} and the fact that
$g_{x_i}^{\mathbb{k}'}=1_G$, since $p_{i}$ divides $\mathbb{k}'$ by
the choice of $\mathbb{k}'$, we deduce
$$\phi\bigl((x_i^{\mathbb{k}'})^2\bigr)=\phi(x_i^{\mathbb{k}+1+\mathbb{q}+\mathbb{k}'})=(x_i^\mathbb{k},g_{x_i}^{\mathbb{q}+\mathbb{k}'},x_i^\mathbb{k})
=(x_i^\mathbb{k},g_{x_i}^{\mathbb{q}},x_i^\mathbb{k})=\phi(x_i^{\mathbb{k}'}).$$
Thus, $\phi(x_i^{\mathbb{k}'})$ is an idempotent power of
$\phi(x_i)$, whence $\phi(x_i^{\omega})=\phi(x_i^{\mathbb{k}'})$.
\end{proof}

Using~\eqref{eq:omega+qi},~\eqref{eq:x_ikq}, Lemma~\ref{lemma:p_iandq'_i}~\ref{item:b} and the fact that $\phi$
and $\check{\mathsf f}$ are homomorphisms such that $\check{\mathsf f}\circ\check{\mathsf f}=\check{\mathsf f}$,
one deduces that
$$\begin{array}{rll}
\phi(\pi)&\hspace*{-2mm}=\phi(u'_0x_1^{\mathbb{k}+{\mathbb
q}_1+1}u_1 x_2^{\mathbb{k}+{\mathbb q}_2+1}\cdots
u_{m-1}x_m^{\mathbb{k}+{\mathbb q}_m+1}u_m)\\

&\hspace*{-2mm}=\varphi\bigl(\check{\mathsf
f}(u'_0x_1^{\mathbb{k}}g_{x_1}^{\mathbb{q}_1}x_1^{\mathbb{k}}u_1
x_2^{\mathbb{k}}g_{x_2}^{\mathbb{q}_2}x_2^{\mathbb{k}}\cdots
u_{m-1}x_m^{\mathbb{k}}g_{x_m}^{\mathbb{q}_m}x_m^{\mathbb{k}}u_m)\bigr)\\

&\hspace*{-2mm}=\varphi\bigl(\grave{\mathsf
f}(u'_0x_1^{\mathbb{k}})g_{x_1}^{\mathbb{q}_1}\hat{\mathsf
f}(x_1^{\mathbb{k}}u_1 x_2^{\mathbb{k}})g_{x_2}^{\mathbb{q}_2}\cdots
\hat{\mathsf
f}(x_{m-1}^{\mathbb{k}}u_{m-1}x_m^{\mathbb{k}})g_{x_m}^{\mathbb{q}_m}\acute{\mathsf
f}(x_m^{\mathbb{k}}u_m)\bigr).
\end{array}$$
Analogously, $\phi(\rho)=\varphi\bigl(\grave{\mathsf
f}(u'_mx_{m+1}^{\mathbb{k}})g_{x_{m+1}}^{\mathbb{q}_{m+1}}\hat{\mathsf
f}(x_{m+1}^{\mathbb{k}}u_{m+1}
x_{m+2}^{\mathbb{k}})g_{x_{m+2}}^{\mathbb{q}_{m+2}}\cdots
\hat{\mathsf
f}(x_{n-1}^{\mathbb{k}}u_{n-1}x_n^{\mathbb{k}})g_{x_n}^{\mathbb{q}_n}\acute{\mathsf
f}(x_n^{\mathbb{k}}u_n)\bigr)$. Let us now prove the following
lemma.
\begin{lemma}\label{lemma:eval_xiuixi+1} Let $i\in I_{m,n}$, $j\in\{1,m+1\}$ and $k\in\{m,n\}$.
\begin{enumerate}
\item\label{item:a} $\hat{\mathsf f}(x_i^{\mathbb{k}}u_i x_{i+1}^{\mathbb{k}})={\mathsf
f}(x_i^{\mathbb{k}}){\mathsf f}(x_i^{\mathbb{j}}u_i
x_{i+1}^{\mathbb{j}}){\mathsf f}(x_{i+1}^{\mathbb{k}})^{t_i+1}$.

\item\label{item:b} If $u'_{j-1}$ is the empty word, then $\grave{\mathsf
f}(u'_{j-1}x_j^{\mathbb{k}})=x_j^{\mathbb{k}}$. Otherwise,
$\grave{\mathsf
f}(u'_{j-1}x_{j}^{\mathbb{k}})=u'_{j-1}x_j^{r_j}x'_j\;{\mathsf
f}(x_j^{\mathbb{k}})$ for some integer $\mathbb{i}\leq
r_j<\mathbb{k}$ and some proper prefix $x'_j$ of $x_j$.

\item\label{item:c} If $u_{k}$ is the empty word, then $\acute{\mathsf f}(x_k^{\mathbb{k}}u_k)=x_k^{\mathbb{k}}$.
Otherwise, $\acute{\mathsf f}(x_{k}^{\mathbb{k}}u_{k})={\mathsf
f}(x_k^{\mathbb{k}})x'_kx_k^{r_k}u_{k}$ for some integer
$\mathbb{i}\leq r_k<\mathbb{k}$ and some proper suffix $x'_k$ of
$x_k$.
\end{enumerate}
\end{lemma}
\begin{proof} For~\ref{item:a}, write $w=x_i^{\mathbb{k}}u_i
x_{i+1}^{\mathbb{k}}$ and  notice that $w\not\in L$, whence
$\mbox{sc}_L[w]$ is of the form $\mbox{sc}_L[w]=(w_0,\ddot
w_1,w_1,\ldots,\ddot w_p,w_p)$ with $p\geq 1$. Moreover, since
$x_i^{\mathbb{k}}{\mathtt i}_1(u_ix_{i+1}), {\mathtt
t}_1(x_iu_i)x_{i+1}^{\mathbb{k}}\not\in L$, $w_0=x_i^{\mathbb{k}}$
and $w_p=x_{i+1}^{\mathbb{k}}$. By definition of ${\mathsf f}$,
${\mathsf f}(\ddot w_j)={\mathsf f}(w_{j'})=1_G$ for every
$j\in\{1,\ldots,p\}$ and $j'\in\{1,\ldots,p-1\}$ such that
$w_{j'}\not\in W_\mathbb{j}\cup W'_\mathbb{k}$. By definition of
$\hat{\mathsf f}$, it results that
\begin{equation}\label{eq:hatfw}
 \hat{\mathsf f}(w)={\mathsf f}(w_0){\mathsf f}(w_{j_1})\cdots {\mathsf f}(w_{j_r}){\mathsf f}(w_p),
\end{equation}
 where $w_0,w_{j_1},\ldots
,w_{j_r},w_p$ (with $0<j_1<\cdots<j_r<p$) is the sequence of all
coordinates of $w$ that belong to the set $W_\mathbb{j}\cup
W'_\mathbb{k}$.

By definition of $t_i$, $x_i^{\mathbb j} u_i=v_ix_{i+1}^{t_i}$ and
$x_{i+1}$ is not a suffix of $v_i$. Hence
$w=x_i^{\mathbb{k}-\mathbb{j}}v_ix_{i+1}^{\mathbb{k}+t_i}$ and
$x_{i+1}^{\mathbb{k}}$ has $t_i+1$ maximal occurrences on the suffix
$x_{i+1}^{\mathbb{k}+t_i}$ of $w$. It is then clear that $r\geq t_i$
and $w_{j_{r-t_i+1}}=\cdots =w_{j_{r}}=w_p=x_{i+1}^{\mathbb{k}}$. On
the other hand, $x_i^{\mathbb{k}}u_i
x_{i+1}^{\mathbb{k}}=x_i^{\mathbb{k}-\mathbb{j}}x_i^{\mathbb{j}}u_i
x_{i+1}^{\mathbb{j}}x_{i+1}^{\mathbb{k}-\mathbb{j}}$ and
$x_i^{\mathbb{j}}u_i x_{i+1}^{\mathbb{j}}\in W_\mathbb{j}$.
Moreover, since, by the choice of $\mathbb{j}$, $x_i^{\mathbb{j}}
 u_ix_{i+1}^{\mathbb{j}}$ is synchronized in $x_j^{-\infty}
 u_jx_{j+1}^{+\infty}$ for every $j\in I_{m,n}$, the words $a_{i\ell_i}x_i^{\mathbb{j}}u_i
x_{i+1}^{\mathbb{j}}$ and $x_i^{\mathbb{j}}u_i
x_{i+1}^{\mathbb{j}}a_{i+1\;\! 1}$ do  not belong to $L$. Hence, the
word $x_i^{\mathbb{j}}u_ix_{i+1}^{\mathbb{j}}\in L$ has a maximal
occurrence in $w$ and so $x_i^{\mathbb{j}}
 u_ix_{i+1}^{\mathbb{j}}=w_{j_s}$ for some $s$ such that $1\leq s\leq r-t_i$. We
 claim that $1=s=r-t_i$ (whence $r=t_i+1$). This claim,
 together with~\eqref{eq:hatfw}, completes the proof of~\ref{item:a}. To show the claim, notice first that the above mentioned occurrence
 of $x_i^{\mathbb{j}}
 u_ix_{i+1}^{\mathbb{j}}$ is the unique occurrence of an element of
 $ W_\mathbb{j}$ in $w$ because these elements are synchronized in $x_i^{-\infty}
 u_ix_{i+1}^{+\infty}$. Suppose, on the other hand, that some element $x_j^\mathbb{k}\in
 W'_\mathbb{k}$ occurs in $w$. Then, by the assumption on $\mathbb{k}$, $x_j^\mathbb{k}$
 overlaps the prefix $x_i^{\mathbb{k}}$ or the suffix $x_{i+1}^{\mathbb{k}}$ of $w$ for a
 factor sufficiently large to deduce (from Fine and
Wilf's Theorem and the fact that $x_i$, $x_{i+1}$ and $x_j$ are Lyndon words)  that $x_j=x_i$ or $x_j=x_{i+1}$.
Since, by definition of canonical form, $x_i$ is not a prefix of $u_ix_{i+1}^{\mathbb j}$ and, as we have seen
above, $x_{i+1}$ is not a suffix of $v_i$, this means that the occurrence of $x_j^\mathbb{k}$ determines one of
the already identified coordinates $w_0,w_{j_{r-t_i+1}},\ldots ,w_{j_r},w_p$ of $w$, thus proving that $j_s=j_1$
and $j_{r-t_i+1}=j_2$. This shows the claim and concludes the proof of~\ref{item:a}.

We now establish~\ref{item:b}. Let $w=u'_{j-1}x_j^{\mathbb{k}}$. If $u'_{j-1}$ is the empty word, then
$w=x_j^{\mathbb{k}}\in L$, whence, by definition of the function $\grave{\mathsf f}$, $\grave{\mathsf
f}(w)=x_j^{\mathbb{k}}$. Suppose that $u'_{j-1}$ is non-empty. Then $w\not\in L$ and, so, $\mbox{sc}_L[w]$ is of
the form $\mbox{sc}_L[w]=(w_0,\ddot w_1,w_1,\ldots,\ddot w_p,w_p)$ with $p\geq 1$. As $x_{j}^{\mathbb{k}}$ is
the longest suffix of $w$ in $L$, $w_p=x_j^{\mathbb{k}}$. On the other hand, since $u'_{j-1}x_j^{\mathbb{i}}\in
W_\mathbb{i}\subseteq L$, the longest prefix of $w$ in $L$ is of the form $u'_{j-1}x_j^{r_j}x'_j$ for some
integer $\mathbb{i}\leq r_j<\mathbb{k}$ and some proper prefix $x'_j$ of $x_j$, whence
$w_0=u'_{j-1}x_j^{r_j}x'_j$. By definition of canonical form and of the set $W_\mathbb{j}\cup W'_\mathbb{k}$ it
is clear that the intermediate coordinates $\ddot w_1,w_1,\ldots,\ddot w_p$ of $w$ do not belong to this set. By
definition of $\grave{\mathsf f}$ it then follows that $\grave{\mathsf f}(w)=w_0{\mathsf
f}(w_p)=u'_{j-1}x_j^{r_j}x'_j\;{\mathsf f}(x_j^{\mathbb{k}})$, thus proving~\ref{item:b}.

Condition~\ref{item:c} is symmetrical to condition~\ref{item:b} and
so it can be proved analogously, thus completing the proof of the
lemma.
\end{proof}

Lemma~\ref{lemma:eval_xiuixi+1}~\ref{item:a} and the definition of
${\mathsf f}$ show that, for every $i\in I_{m,n}$,
$$\hat{\mathsf f}(x_i^{\mathbb{k}}u_i
x_{i+1}^{\mathbb{k}})=g_{x_i}g_{x_i}^{-1}g_{x_i,u_i,x_{i+1}}g_{x_{i+1}}^{-t_i-1}g_{x_{i+1}}^{t_i+1}=g_{x_i,u_i,x_{i+1}}.$$
Combining this with conditions~\ref{item:b} and~\ref{item:c} of Lemma~\ref{lemma:eval_xiuixi+1} and with the
previous calculations, we may now finish the evaluation of $\pi$ and $\rho$ in $S_{\pi,\rho}$. Let $z'_0$ be the
first coordinate of $u'_{0}x_1^{\mathbb{k}}$ and $z_m$ be the last coordinate of $x_m^{\mathbb{k}}u_m$
determined by $L$. Let $h'_0=g_{x_1}$  if $u'_{0}\neq 1$ and  $h'_0=1_G$ otherwise. Analogously, let
$h_m=g_{x_m}$   if
  $u_{m}\neq 1$ and  $h_m=1_G$ otherwise. Then
$$\begin{array}{rll}
\phi(\pi)
&\hspace*{-2mm}=\varphi(z'_0h'_0g_{x_1}^{\mathbb{q}_1}g_{x_1,u_1,
x_{2}}g_{x_2}^{\mathbb{q}_2}\cdots
g_{x_{m-1},u_{m-1},x_{m}}g_{x_m}^{\mathbb{q}_m}h_mz_m)\\

&\hspace*{-2mm}=(z'_0,h'_0g_{x_1}^{\mathbb{q}_1}g_{x_1,u_1,
x_{2}}g_{x_2}^{\mathbb{q}_2}\cdots
g_{x_{m-1},u_{m-1},x_{m}}g_{x_m}^{\mathbb{q}_m}h_m,z_m)\\

&\hspace*{-2mm}=(z'_0,g_\pi,z_m)
\end{array}$$
and, analogously, $\phi(\rho)=(z'_m,g_\rho,z_n)$. Since $S_{\pi,\rho}$ verifies the $\kappa$-identity
$\pi=\rho$,  by hypothesis, we have that $\phi(\pi)=\phi(\rho)$. Then $z'_0=z'_m$, $g_\pi=g_\rho$ and $z_m=z_n$.
From  $g_\pi=g_\rho$ we deduce that the identity ${\mathsf w}_\pi={\mathsf w}_\rho$ is trivial. In particular,
$\mathsf{u}'_{0}\mathsf{v}_{x_1}^{\mathbb{q}_1}=\mathsf{u}'_{m}\mathsf{v}_{x_{m+1}}^{\mathbb{q}_{m+1}}$ and, so,
$x_1=x_{m+1}$. Now, it results from the equality $z'_0=z'_m$ and the fact that $\pi$ and $\rho$ are in canonical
form that $u'_0=u'_m$. By symmetry, we also have $u_m=u_n$. This completes the proof that~\ref{item:c}
implies~\ref{item:d}.
\end{proof}

The following decidability result may now be easily deduced.
\begin{corollary}\label{corol:LGkappa_word_problem_dec_rank1}
 Given $\kappa$-terms $\pi$ and $\rho$ of rank at most 1, it is decidable whether ${\bf LG}$ (resp.\ $\Se$) satisfies $\pi=\rho$.
\end{corollary}
\begin{proof} Since $\LG$ contains the
pseudovariety $\LI$ and $\LI$ separates
 different finite words as well as finite words from rank 1 $\kappa$-terms,  if one of $\pi$ and $\rho$ is a finite word
then they both are the same finite word. So, we assume that
 $\pi$ and $\rho$ are both rank 1 \kt s. In view of Theorem~\ref{theo:wp_rank1}, to decide whether $\LG$
(resp.\ $\Se$) verifies $\pi=\rho$ it suffices to compute the canonical forms $\pi'$ and $\rho'$ of $\pi$ and
$\rho$, respectively, and to verify if $\pi'$ and $\rho'$ are the same $\kappa$-term.
\end{proof}

\section{An alternative version of Theorem~\ref{theo:wp_rank1} }\label{section:alternative proof}
By way of comparison, we present an alternative version of Theorem~\ref{theo:wp_rank1}. This second version
contains some adjustments in the construction of the group identity ${\mathsf w}_\pi={\mathsf w}_\rho$ and of
the test-semigroup $S_{\pi,\rho}$ associated with a given $\kappa$-identity $\pi=\rho$. Although
in~\cite{Costa&Nogueira&Teixeira:2013} we will follow the first scheme to complete the proof of the decidability
of the word problem for $\kappa$-terms over ${\bf LG}$, this new construction is less tricky and has some
advantages that may be exploited in future work. Here, the test-semigroup $S_{\pi,\rho}$ is of the form
$S_k(G,{\mathsf f})$ and the group identity ${\mathsf w}_\pi={\mathsf w}_\rho$ involves $-1$ exponents and is not reduced in the free group. Instead of $S_k(G,{\mathsf f})$, we could use a semigroup $M_k(G,{\hbar})$ with exactly the same effect. Actually, the group identity ${\mathsf w}_\pi={\mathsf w}_\rho$ is the one that would emerge if one would use the representation of the free pro-$(\G*\D_k)$ semigroups obtained by Almeida and Azevedo~\cite{Almeida&Azevedo:1993}.

We let
$\pi=u'_0x_1^{\omega+q_1}u_1 \cdots x_m^{\omega+q_m}u_m$ and $\rho=u'_mx_{m+1}^{\omega+q_{m+1}}u_{m+1} \cdots
x_n^{\omega+q_n}u_n$ be rank 1 \kt s in canonical form and assume the same notations $I_n=\{1,\ldots,n\}$,
$I_{m,n}=I_n\setminus\{ m,n\}$, $\ell_i=|x_i|$, $x_i=a_{i1}a_{i2}\cdots a_{i\ell_i}$ and $x_{ij}=a_{ij}\cdots
a_{i\ell_i}a_{i1}\cdots a_{i\;\!j-1}$ of the previous section.  Now, let $\mathbb{k}\in\N$ be such that
$\mathbb{k}+1=r\ell_1\ell_2\cdots \ell_n$, with $r>1+\mbox{max}\{|u'_0|,|u'_m|, |u_j|,|q_j|:j\in I_n\}$, and let
$r_i=r\ell_1\cdots\ell_{i-1}\ell_{i+1}\cdots \ell_n=\frac{\mathbb{k}+1}{\ell_i}$. Denote at last
$b_{ij}=x_{ij}^{r_i}$, $b'_{i1}={\mathtt i}_\mathbb{k}(b_{i1})=x_{i}^{r_i-1}a_{i1}a_{i2}\cdots
a_{i\;\!\ell_i-1}$ and $b''_{i1}={\mathtt t}_\mathbb{k}(b_{i1})=a_{i2}a_{i3}\cdots
a_{i\;\!\ell_i}x_{i}^{r_i-1}$.

We define first the group identity ${\mathsf w}_\pi={\mathsf
w}_\rho$. Let  $L_\mathbb{k}=A^{\leq \mathbb{k}}$, whence $\ddot
L_\mathbb{k}=A^{\mathbb{k}+1}$. We associate to each word $u\in
\ddot L_\mathbb{k}$ a variable ${\mathsf v}_u$, let ${\mathsf
V}=\{{\mathsf v}_u\mid u\in \ddot L_\mathbb{k}\}$ and denote by
$F_{\mathsf V}$ the free group generated by the alphabet ${\mathsf
V}$. Let $\lambda:A^+\rightarrow F_{\mathsf V}$ be the
$\mathbb{k}$-superposition homomorphism given by
$\lambda(u)={\mathsf v}_u$ for every $u\in \ddot L_\mathbb{k}$,
whence  for a word $v\in A^+\setminus L_\mathbb{k}$,
$\lambda(v)={\mathsf v}_{v_1}{\mathsf v}_{v_2}\cdots {\mathsf
v}_{v_{|v|-\mathbb{k}}}$ where
$v_1,v_2,\ldots,v_{|v|-\mathbb{k}}\in\ddot L_{\mathbb k}$ are the
successive factors of length $\mathbb{k}+1$ of $v$. Now, let
$${\mathsf w}_\pi={\mathsf w}'_0{\mathsf y}_1^{q_1-r_1}{\mathsf w}_1 \cdots {\mathsf y}_m^{q_m-r_m}{\mathsf w}_{m}\quad\mbox{and}\quad
 {\mathsf w}_\rho={\mathsf w}'_{m}{\mathsf y}_{m+1}^{q_{m+1}-r_{m+1}}{\mathsf w}_{m+1} \cdots
{\mathsf y}_n^{q_n-r_n}{\mathsf w}_n$$ where, for $i\in I_n$ and
$j\in\{0,m\}$, ${\mathsf y}_i=\lambda(x_{i1}b'_{i1})$, ${\mathsf
w}'_j=\lambda(u'_jb'_{j+1\;\!j+1})$, ${\mathsf
w}_i=\lambda(b_{i1}u_i)$ if $i\in\{m,n\}$ and ${\mathsf
w}_i=\lambda(b_{i1}u_ib'_{i+1\;\!1})$ otherwise. We denote by
$\widetilde{\mathsf w}_\pi$ and $\widetilde{\mathsf w}_\rho$,
respectively, the reduced forms of ${\mathsf w}_\pi$ and ${\mathsf
w}_\rho$ in $F_{\mathsf V}$.

Let us now define the test-semigroup $S_{\pi,\rho}$. Choose a group homomorphism $\eta:F_{\mathsf V}\rightarrow
G$ into a finite group $G$  such  that $\eta({\mathsf w}_\pi)\neq \eta({\mathsf w}_\rho)$ when
$\widetilde{\mathsf w}_\pi\neq \widetilde{\mathsf w}_\rho$. Then let ${\mathsf f}:L_\mathbb{k}\cup\ddot
L_\mathbb{k}\rightarrow G$ be the mapping defined by ${\mathsf f} (L_\mathbb{k})=\{1_G\}$ and ${\mathsf f}
(u)=\eta({\mathsf v}_{u})=\eta(\lambda(u))$ for each $u\in \ddot L_\mathbb{k}$. Finally, we let
$S_{\pi,\rho}=S_\mathbb{k}(G,{\mathsf f})$. Recall that  condition ${\mathsf f} (L_\mathbb{k})=\{1_G\}$ makes
the function $\hat{\mathsf f}$ define a $\mathbb{k}$-superposition homomorphism from $A^+$ into $G$ (the
$\mathbb{k}$-superposition property is a particular case of the property presented in
Lemma~\ref{lemma:property_hat_function}~\ref{eq:property_hat_function}).

The announced alternative version of Theorem~\ref{theo:wp_rank1} is the following.
\begin{theorem}\label{theo2:wp_rank1}
Under the above assumptions and with the above notations, the following conditions are equivalent:
\begin{enumerate}
\item\label{item:a} $\Se\models\pi=\rho$.

\item\label{item:b} $\LG\models\pi=\rho$.

\item\label{item:c} $S_{\pi,\rho}\models\pi=\rho$.

\item\label{item:d} $\widetilde{\mathsf w}_\pi=\widetilde{\mathsf w}_\rho$, $u'_0=u'_m$ and
$u_m=u_n$.

\item\label{item:e} $\pi$ and $\rho$ are the same \kt.
\end{enumerate}
\end{theorem}
\begin{proof} The implications~\ref{item:e}$\Rightarrow$\ref{item:a}$\Rightarrow$\ref{item:b}$\Rightarrow$\ref{item:c} hold
trivially. To prove the implication~\ref{item:c}$\Rightarrow$\ref{item:d}, assume that
$S_{\pi,\rho}\models\pi=\rho$ and let  $\phi_\mathbb{k}: A^+\rightarrow S_\mathbb{k}(G,{\mathsf f})$ be the
homomorphism $\phi_\mathbb{k}=\varphi\circ \check{\mathsf f}$, where $\varphi$ is the isomorphism from
$\mathcal{Z}[G,L_\mathbb{k},{\mathsf f}]$ onto $S_{\pi,\rho}$ defined in
Section~\ref{Presentations-local-groups}, and notice that $\phi_\mathbb{k}(w)=({\mathtt
i}_\mathbb{k}(w),\hat{\mathsf f}(w),{\mathtt t}_\mathbb{k}(w))$ for $w\in A^+\setminus L_\mathbb{k}$. We use the
same notation for the homomorphism of $\kappa$-semigroups $T_A^\kappa\rightarrow S_{\pi,\rho}$ that coincides
with $\phi_\mathbb{k}$ on $A^+$.

 Since $S_{\pi,\rho}$ is a finite semigroup, it verifies the
identity $x^{p!}=x^\omega$ for every sufficiently large positive integer $p$. Take one such $p$ with
$p!+q_i\geq\mathbb{k}+1$ for all $i\in I_n$. Then, $\phi_\mathbb{k}(\pi)=\phi_\mathbb{k}(\pi_p)$ and
$\phi_\mathbb{k}(\rho)=\phi_\mathbb{k}(\rho_p)$ where $\pi_p$ and $\rho_p$ are the words from  $A^+$ obtained,
respectively, from $\pi$ and $\rho$ by replacing each occurrence of $\omega$ by $p!$. Since $\hat{\mathsf f}$ is
a $\mathbb{k}$-superposition homomorphism, $\phi_\mathbb{k}(x_i^{p!+q_i})=(b'_{i1},({\mathsf f}(b_{i1}){\mathsf
f}(b_{i2})\cdots {\mathsf f}(b_{i\ell_i}))^{p!+q_i-r_i}{\mathsf f}(b_{i1}),b''_{i1})$. On the other hand
$({\mathsf f}(b_{i1}){\mathsf f}(b_{i2})\cdots {\mathsf f}(b_{i\ell_i}))^{p!+q_i-r_i}=({\mathsf
f}(b_{i1}){\mathsf f}(b_{i2})\cdots {\mathsf f}(b_{i\ell_i}))^{q_i-r_i}$ because $G$ is a subgroup of
$S_{\pi,\rho}$ and, so, it verifies the identity $x^{p!}=1$. Using again the fact that $\hat{\mathsf f}$ is a
$\mathbb{k}$-superposition homomorphism, it follows that
$$\phi_\mathbb{k}(\pi)=\phi_\mathbb{k}(\pi_p)=({\mathtt i}_\mathbb{k}(u'_0b_{11}),w'_0y_1^{q_1-r_1}w_1 \cdots
 y_m^{q_m-r_m}w_{m},{\mathtt t}_\mathbb{k}(b_{m1}u_m))$$
where $y_i={\mathsf f}(b_{i1}){\mathsf f}(b_{i2})\cdots {\mathsf f}(b_{i\ell_i})=\hat{\mathsf
f}(x_{i1}b'_{i1})$, $w'_0=\hat{\mathsf f}(u'_0b'_{11})$, $w_m={\mathsf f}(b_{m1})\hat{\mathsf
f}(b''_{m1}u_m)=\hat{\mathsf f}(b_{m1}u_m)$ and, for $i\neq m$, $w_i={\mathsf f}(b_{i1})\hat{\mathsf
f}(b''_{i1}u_ib'_{i+1\;\!1})=\hat{\mathsf f}(b_{i1}u_ib'_{i+1\;\!1})$. Analogously,
$$\phi_\mathbb{k}(\rho)=\phi_\mathbb{k}(\rho_p)=({\mathtt i}_\mathbb{k}(u'_mb_{m+1\;\!1}),
w'_my_{m+1}^{q_{m+1}-r_{m+1}}w_{m+1} \cdots y_n^{q_n-r_n}w_{n},{\mathtt t}_\mathbb{k}(b_{n1}u_n)).$$

As one may confirm easily, the following equalities hold
$$\begin{array}{l}
\eta({\mathsf w}_\pi)=w'_0y_1^{q_1-r_1}w_1 \cdots y_m^{q_m-r_m}w_{m}=\hat{\mathsf f}(\pi_p),\\
\eta({\mathsf w}_\rho)=w'_my_{m+1}^{q_{m+1}-r_{m+1}}w_{m+1} \cdots y_n^{q_n-r_n}w_{n}=\hat{\mathsf f}(\rho_p).
\end{array}$$
By hypothesis $S_{\pi,\rho}$ verifies $\pi=\rho$, whence $\phi_\mathbb{k}(\pi)=\phi_\mathbb{k}(\rho)$. Thus
${\mathtt i}_\mathbb{k}(u'_0b_{11})={\mathtt i}_\mathbb{k}(u'_mb_{m+1\;\!1})$, $\eta({\mathsf
w}_\pi)=\eta({\mathsf w}_\rho)$ and ${\mathtt t}_\mathbb{k}(b_{m1}u_m)={\mathtt t}_\mathbb{k}(b_{n1}u_n)$. The
second condition $\eta({\mathsf w}_\pi)=\eta({\mathsf w}_\rho)$ and the definition of $\eta$ imply  that
$\G\models {\mathsf w}_\pi={\mathsf w}_\rho$ and, so, that the identity $\widetilde{\mathsf
w}_\pi=\widetilde{\mathsf w}_\rho$ is trivial. In particular ${\mathsf w}'_0{\mathsf y}_1^{q_1-r_1}={\mathsf
w}'_m{\mathsf y}_{m+1}^{q_{m+1}-r_{m+1}}$. Since $r$ may be chosen arbitrarily large the same happens with every
$r_i$. Hence, each exponent $q_i-r_i$ is negative, with absolute value as large as we want. Therefore, as $\pi$
and $\rho$ are in canonical form, one deduces easily that ${\mathsf w}'_0={\mathsf w}'_m$
 and ${\mathsf y}_1={\mathsf y}_{m+1}$ and then that $x_1=x_{m+1}$. Now, the first condition ${\mathtt
i}_\mathbb{k}(u'_0b_{11})={\mathtt i}_\mathbb{k}(u'_mb_{m+1\;\!1})$ above immediately implies that $u'_0=u'_m$.
Symmetrically, one deduces that $u_m=u_n$. This shows the implication~\ref{item:c}$\Rightarrow$\ref{item:d}.

Let us now prove that~\ref{item:d} implies \ref{item:e}. Assume that  $\widetilde{\mathsf
w}_\pi=\widetilde{\mathsf w}_\rho$ is a trivial identity and that $u'_0=u'_m$ and  $u_m=u_n$. For $j\in I_{m,n}$
let $u_j=x'_jv''_j=v'_jx''_{j+1}$ be two factorizations of $u_j$ where $x'_j$ (resp.\ $x''_{j+1}$) is the
longest common prefix (resp.\ suffix) of $u_j$ and $x_j$ (resp.\ $x_{j+1}$).  Consider  the words
$z''_j={\mathtt t}_{\mathbb{k}-|x'_j|}(x_j^\mathbb{k})$ and $z'_{j+1}={\mathtt
i}_{\mathbb{k}-|x''_{j+1}|}(x_{j+1}^\mathbb{k})$. With this notation and since $\lambda$ is a
$\mathbb{k}$-superposition homomorphism, we have
$$\begin{array}{rl}
{\mathsf w}_j&\hspace*{-3mm}=\lambda(b_{j1}x'_jv''_jb'_{j+1\;\!1})=\lambda(b_{j1}x'_{j})\lambda({\mathtt t}_\mathbb{k}(b_{j1}x'_{j})v''_j b'_{j+1\;\!1})\\[1mm]
 &\hspace*{-3mm}=\lambda(b_{j1}x'_{j})\lambda(z''_jx'_jv''_j b'_{j+1\;\!1})=\lambda(b_{j1}x'_{j})\lambda(z''_jv'_jx''_{j+1} b'_{j+1\;\!1})\\[1mm]
&\hspace*{-3mm}=\lambda(b_{j1}x'_{j})\lambda(z''_ju_jz'_{j+1})\lambda(x''_{j+1}b'_{j+1\;\!1}).
\end{array}$$
By the choice of $x'_j$, $\lambda(b_{j1}x'_{j})$ is the longest prefix of ${\mathsf w}_j$ that is canceled by
${\mathsf y}_j^{q_j-r_j}$ in the reduction process of ${\mathsf w}_\pi$ (if $j<m$) and ${\mathsf w}_\rho$
(otherwise). Dually, $\lambda(x''_{j+1}b'_{j+1\;\!1})$ is the longest suffix of ${\mathsf w}_j$ that is canceled
by ${\mathsf y}_{j+1}^{q_{j+1}-r_{j+1}}$. Moreover, $\lambda(b_{j1}x'_{j})$ and
$\lambda(x''_{j+1}b'_{j+1\;\!1})$ are words over the alphabet ${\mathsf V}$ of length, respectively, $|x'_j|+1$
and $|x''_{j+1}|$ and so less than or equal to $|u_j|$. Since each exponent $q_i-r_i$ is negative and with a
large absolute value, we may assume that some (large) factor ${\mathsf y}'_j$ of ${\mathsf y}_j^{q_j-r_j}$
remains in the reduced form. On the other hand, $\lambda(z''_ju_jz'_{j+1})\neq 1$ since
$|z''_ju_jz'_{j+1}|>\mathbb{k}$. Whence, the reduced forms $\widetilde{\mathsf w}_\pi$ and $\widetilde{\mathsf
w}_\rho$ have precisely, respectively, $m-1$ and $n-m-1$ factors of ${\mathsf V}^+$ that are placed between
factors of $({\mathsf V}^{-1})^+$. Since $\widetilde{\mathsf w}_\pi=\widetilde{\mathsf w}_\rho$ by hypothesis,
it follows that $m=n-m$ and, for every $j\in\{1,\ldots,m-1\}$,
$\lambda(z''_ju_jz'_{j+1})=\lambda(z''_{m+j}u_{m+j}z'_{m+j+1})$, whence
$z''_ju_jz'_{j+1}=z''_{m+j}u_{m+j}z'_{m+j+1}$. As $\mathbb{k}$ (and, so, also the words $z''_j$, $z'_{j+1}$,
$z''_{m+j}$ and $z'_{m+j+1}$) is arbitrarily large, one deduces that the bi-infinite words
$\infee{x}{j}u_j\infdd{x}{j+1}$ and $\infee{x}{m+j}u_{m+j}\infdd{x}{m+j+1}$ coincide. Hence, as a consequence of
the hypothesis  $u_m=u_n$ and of the rank 1 canonical form definition,  for every $j\in\{1,\ldots,m\}$,
$x_j=x_{m+j}$ and $u_j=u_{m+j}$, and thus $w_j=w_{m+j}$, $y_j=y_{m+j}$ and $r_j=r_{m+j}$. To
 deduce that $\pi$ and $\rho$ coincide, and since by hypothesis $u'_0=u'_m$, it remains to prove that $q_j=q_{m+j}$
 for every $j\in\{1,\ldots,m\}$. But that is immediate after the above equalities and the fact that  $\widetilde{\mathsf w}_\pi=\widetilde{\mathsf
 w}_\rho$. This shows the implication~\ref{item:d}$\Rightarrow$\ref{item:e} and concludes the proof of Theorem~\ref{theo2:wp_rank1}.
\end{proof}

\paragraph{\textbf{Final remarks.}}

The semigroups $\mathcal{S}(G,L,{\mathsf f})$ were used above as
test-semigroups to show that the word problem for $\kappa$-terms of
rank at most 1 over ${\bf LG}$ is decidable, and they
 are used in~\cite{Costa&Nogueira&Teixeira:2013} to extend this result to any
 rank. Exploring the fact that $\LH$ is generated by the semigroups $\mathcal{S}(G,L,{\mathsf f})$  with
$G\in\He$ and $L$ finite, we expect that the above method can be
adapted to solve the $\kappa$-word problem over $\LH$, for other
pseudovarieties $\He$ of groups.

\section*{Acknowledgments}
This work was supported by the European Regional Development Fund, through the programme COMPETE, and by the
Portuguese Government through FCT -- \emph{Funda\c c\~ao para a Ci\^encia e a Tecnologia}, under the project
PEst-C/MAT/UI0013/2011.


\end{document}